\documentclass[11pt]{article}
\usepackage{amsmath}
\usepackage{amsfonts}
\usepackage{amssymb}
\usepackage[english]{babel}
\usepackage{graphicx}
\usepackage{amsthm}
\usepackage{hyperref}
\usepackage{accents}
\usepackage [autostyle, english = american]{csquotes}
\MakeOuterQuote{"}

\numberwithin{equation}{section}
\theoremstyle{plain}
\newtheorem*{theorem*}{Theorem}
\newtheorem*{lemma*}{Lemma}
\newtheorem{theorem}{Theorem}
\newtheorem{lemma}{Lemma}[section]
\newtheorem{corollary}[lemma]{Corollary}

\newtheorem{proposition}[lemma]{Proposition}

\theoremstyle{definition}

\newtheorem{remark}[lemma]{Remark}

\def\pr{\partial}
  
\def\V{\Vert}

\usepackage{geometry}
\begin{document}

\title{On the Effects of Advection and Vortex Stretching}
\author{Tarek Mohamed Elgindi and In-Jee Jeong}
\date{\today}

\maketitle

\begin{abstract}
We prove finite-time singularity formation for De Gregorio's model of the three-dimensional vorticity equation in the class of $L^p\cap C^\alpha(\mathbb{R})$ vorticities for some $\alpha>0$ and $p<\infty$. We also prove finite-time singularity formation from smooth initial data for the Okamoto-Sakajo-Wunsch models in a new range of parameter values. As a consequence, we have finite-time singularity for certain infinite-energy solutions of the surface quasi-geostrophic equation which are $C^\alpha$-regular. One of the difficulties in the models we consider is that there are competing \emph{nonlocal} stabilizing effects (advection) and destabilizing effects (vortex stretching) which are of the same size in terms of scaling. Hence, it is difficult to establish the domination of one effect over the other without having strong control of the solution.  

We conjecture that strong solutions to the De Gregorio model exhibit the following behavior: for each $0<\alpha<1$ there exists an initial $\omega_0\in C^\alpha(\mathbb{R})$ which is compactly supported for which the solution becomes singular in finite-time; on the other hand, solutions to De Gregorio's equation are global whenever $\omega_0\in L^p\cap C^{1}(\mathbb{R})$ for some $p<\infty$. Such a dichotomy seems to be a genuinely non-linear effect which cannot be explained merely by scaling considerations since $C^\alpha$ spaces are scaling subcritical for each $\alpha>0$. 
\end{abstract}

\tableofcontents

\section{Introduction}

\subsection{Synopsis}
In 1985 Constantin, Lax, and Majda introduced a one-dimensional equation which models the phenomenon of vortex-stretching in a three-dimensional incompressible fluid. They established finite-time singularity formation in their model. Following in their footsteps, De Gregorio introduced a one-dimensional equation which models the effect of \emph{both} transport and vortex stretching. De Gregorio, aided by some numerical simulations, then conjectured that his introduction of velocity transport leads to global regularity. Thereafter, Okamoto, Sakajo, and Wunsch introduced a continuum of models which give different "weights" to the vortex stretching and transport terms by introducing a real parameter $a$ in front of the transport term. The full model then became: $$\partial_t\omega+ a \, u\partial_x\omega=\omega\partial_x u$$
$$u=(-\Delta)^{-1/2}\omega,$$ with the Constantin-Lax-Majda and De Gregorio models being the cases $a=0$ and $a=1$ respectively. They conjectured that when the transport term is strictly weaker than the vortex stretching term $(a<1)$, finite time singularities would occur and that there would be global regularity when the transport term is as strong or stronger than the vortex stretching term $(a\geq 1)$. This was previously unconfirmed except in the case $a\leq 0$ where the transport term and vortex stretching term are actually working together toward singularity formation. All of this revolves around the notion of the stabilizing effect of transport and the destabilizing effect of vortex stretching.   In this work, we will prove singularity formation in the range $|a|<a_0$ for some $a_0>0$ from initially smooth profiles. To this end, we first prove the existence of self-similar finite-time singularity formation from smooth data for the Constantin-Lax-Majda equation. Then we use a "soft" continuation argument to prove singularity formation when $|a|$ is small. One of the interesting features of this construction is that the solutions, while becoming singular in finite time, remain uniformly bounded in a norm which is strictly stronger than those predicted by the scaling of the equation and available conservation laws -- a phenomenon which is conjectured to happen when $a=-1$ (\cite{CCF}, \cite{SV}, \cite{Do}). Finally, using merely H\"older continuous self-similar solutions for the Constantin-Lax-Majda equation, we establish singularity formation for De-Gregorio's model (and for all $a\in\mathbb{R}$) in classes of H\"older continuous solutions. Furthermore, for each $\epsilon>0,$ we establish that such self-similar solutions lead to singularities in finite time which keep the velocity field bounded in the $C^{1-\epsilon}$ norm near the singular point though the velocity's $C^1$ norm is becoming infinite. This relates to the $C^\frac{1}{2}$ conjecture for the Cordoba-Cordoba-Fontelos model. Hence, it seems that the De Gregorio model exhibits two types of behavior for strong solutions: first there are "somewhat smooth" solutions for which there is singularity formation in finite time, then there are "very smooth" solutions and it is conjectured that these remain smooth for all time. Such a phenomenon is also neither captured by the scaling of the equation nor conservation of any critical quantity. 

Our interest in these models is two-fold: first, they give a simplified setting to understand the complex interaction between transport and vortex stretching. Second, as the authors have shown in \cite{EJ1}, these models actually serve as the equations for the evolution of scale-invariant solutions to the SQG equation and similar active scalar systems, which are locally well-posed in function spaces containing scale-invariant solutions (with some symmetry assumption on the initial data). 

\subsection{The incompressible Euler equation}

The main difference between the two-dimensional incompressible Euler equation and its three-dimensional counterpart is the presence of the so-called vortex-stretching term in three dimensions. The two-dimensional incompressible Euler equation can be written as: $$\partial_t\omega+u\cdot\nabla\omega=0,$$ $$u=\nabla^\perp (-\Delta)^{-1}\omega,$$ while the three-dimensional Euler equation can be written as: $$\partial_t\omega+u\cdot\nabla\omega=\omega\cdot\nabla u,$$ $$u=\nabla\times(-\Delta)^{-1}\omega.$$ In two dimensions, $\omega$ is a scalar quantity and is advected by a velocity vector field which it induces. In three dimensions, $\omega$ is a vector quantity. The term $u\cdot\nabla\omega$ is a transport term and, since $u$ is divergence free, it cannot cause any growth of the vorticity. The term $\omega\cdot\nabla u$ is called the vortex stretching term and can lead to amplification of the vorticity. As is shown in the classical work of Beale, Kato, and Majda \cite{BKM}, finite-time singularity formation for strong solutions of the incompressible Euler equation happens if and only if the sup-norm of the vorticity becomes infinite.  

The absence of vortex stretching in two dimensions allows the two dimensional Euler equation to enjoy infinitely many coercive conserved quantities, notably all $L^p$ norms of the vorticity for $1\leq p\leq \infty$. The presence of vortex stretching makes the three-dimensional Euler equation unstable with respect to all $L^p$ norms which were conserved in the two dimensional case and leads to point-wise exponential or double-exponential growth of the vorticity in time (see, for example, \cite{KS},\cite{EM1}). Despite the clear destabilizing properties of the vortex stretching term, the global regularity question of the three-dimensional Euler equation is wide open. In our view, there are three principle reasons why the regularity question is difficult:
\begin{enumerate}
\item The quadratic and non-local nature of the vortex stretching term.
\item The presence of \emph{both} the vortex stretching and the transport terms.
\item The 3D Euler equation is a \emph{system} rather than a scalar equation. 
\end{enumerate}
The importance of the second difficulty is discussed in great detail in \cite{HL}. In view of these many difficulties, experts in the field have devised simpler model equations that share some of the above properties for which more can be proven. 

\subsection{A short history of 1D Models}
\subsubsection{The Starting Point: Constantin, Lax, and Majda}
The first attempt to derive a simplified 1D equation which models the effects of vortex stretching seems to be that of Constantin, Lax, and Majda \cite{CLM} where they wished to consider a non-linear evolution equation whose non-linearity is quadratic and non-local just like the vortex stretching term $\omega\cdot\nabla u$. Considering the exact relation between $u$ and $\omega$ given by $u=\nabla\times(-\Delta)^{-1}\omega$, one sees that $\nabla u$ is just a matrix of singular integral operators applied to $\omega$. Hence, one has the relation:
$$\omega\cdot\nabla u=R(\omega)\omega$$ where $R$ is some matrix of singular integrals. To extract a simple model from this, Constantin, Lax, and Majda dropped the transport term, made the 3D Euler system into a one-dimensional scalar equation, and replaced $R$ by the Hilbert transform. This led to:
$$\partial_t\omega=\omega H(\omega)$$ where $H$ is the Hilbert transform. This equation can then be solved \emph{exactly} using some special properties of the Hilbert transform -- namely, the Tricomi identity (see Appendix A.1).  The reader should notice that in our above breakdown of the difficulties associated with the 3D Euler equation, the Constantin-Lax-Majda equation only takes into account the first of the difficulties: the vortex stretching term is quadratic and non-local.  
\subsubsection{The De Gregorio Model: The Regularity Conjecture}

One of the unfavorable properties of the reduction of Constantin, Lax, and Majda is that it does not behave well with the presence of diffusion. In fact, as Schochet has shown \cite{Sch}, the viscous CLM equation can experience finite time singularity formation \emph{before} its inviscid counterpart even for small viscosity. More than this, the model does not take into account the effects of the transport term. To fix this problem, De Gregorio (\cite{DG1},\cite{DG2}) introduced a variant of the CLM equation that is much closer to the 3D Euler equation, which even has a large set of spatially periodic steady states as the 3D Euler equation does. De Gregorio's model is:
$$\partial_t\omega+u\partial_x\omega=\omega\partial_x u$$
$$\partial_x u =H(\omega).$$
Notice that De Gregorio kept the same vortex stretching term but added a transport term as well. It is important to note that any function of the form $a_n \sin(nx)+b_n\cos(nx)$ is a stationary solution to De-Gregorio's equation while there are no non-trivial stationary states to the CLM equation\footnote{This is because any mean-zero periodic steady state $\omega$ would have to satisfy: $\omega H(\omega)=0$ and $\omega^2-H(\omega)^2=0$ which would imply that $\omega\equiv 0$. When $\omega$ has non-zero mean, consider $\omega=\tilde\omega+c$ with $\tilde\omega$ mean-zero and so $\tilde\omega H(\tilde\omega)+cH(\tilde\omega)=0$ and $H(\tilde\omega)^2-\tilde\omega^2=-2c\tilde\omega.$ Notice that if $A$ is $0$ somewhere, $B$ is also $0$ there. In other regions, we could have $B=-c$ but since $B$ must be mean-zero and $B$ only takes on the values $-c$ and $0$, $B\equiv 0.$ This rules out the existence of any non-constant periodic $L^2$ stationary solution of the CLM equation.}. Moreover, if one considers the CLM equation with data $\omega_0=\sin(x)$, the solution becomes singular in finite time -- developing a singularity like $\frac{1}{x}$ near $x=0$ as $t\rightarrow 2^-$ (see next section), but if one considers the De Gregorio equation with the same data, there is no singularity formation ($\sin(x)$ is stationary!). This already indicates that the transport term can have a stabilizing effect. Numerical simulations done by De Gregorio \cite{DG2}, Okamoto, Sakajo, and Wunsch \cite{OSW}, and others indicate, in fact, that singularity formation formation for De Gregorio's model is impossible from smooth data. In this work, we prove that singularity formation for De Gregorio's model is possible in spaces of H\"older-continuous functions. This seems to be the first rigorous large-data\footnote{To our knowledge, the only existing results were local well-posedness results (see \cite{OSW}).} result for De Gregorio's model. 

\subsubsection{The Cordoba-Cordoba-Fontelos Model: The $C^{\frac{1}{2}}$-conjecture.}

In the spirit of Constantin, Lax, and Majda's idea of getting a simple 1D model of the 3D Euler equation, Cordoba, Cordoba, and Fontelos found a 1D model for the so-called surface quasi-geostrophic (SQG) equation.\footnote{The SQG equation was introduced in the works of Constantin, Majda, and Tabak \cite{CMT1,CMT2} as a mathematical model for geophysical atmospheric flows, and also as a model for the vorticity dynamics in 3D Euler equations -- see for instance the book of Majda and Bertozzi \cite{MB}.} Recall the surface quasi-geostrophic equation:
$$\partial_t\theta+u\cdot\nabla\theta=0,$$
$$u=R^\perp \theta,$$ with $R^\perp=(R_2,-R_1)$ where $R_i$ are the Riesz transforms on $\mathbb{R}^2$. Cordoba, Cordoba, and Fontelos modeled this by the following 1D equation:
$$\partial_t \theta+u\partial_x\theta=0,$$
$$u=H(\theta).$$
Upon differentiating this equation and setting $\partial_x\theta=\omega$ we get:
$$\partial_t\omega+u\partial_x\omega+\omega\partial_x u=0$$
$$\partial_x u=H(\omega)$$
which is the same as De Gregorio's model except that the vortex stretching term has the opposite sign. 

Cordoba, Cordoba, and Fontelos  proved in \cite{CCF} finite time singularity formation for this model, again, using some special properties of the Hilbert transform. Unlike the case with the Constantin-Lax-Majda model, they were not able to solve the equation exactly but were able to prove that certain kinds of smooth initial data cannot have global-in-time solutions by proving that $$\partial_t X\geq cX^2$$ for some positive quantity $X$ which satisfies $X\leq \V\partial_x\theta\V_{L^\infty}$ and some constant $c>0$.  Thereafter, authors tried to push this blow-up result to prove that $\theta$ actually develops a shock at the time of blow-up. However, numerics indicate that a different phenomenon than shock formation is taking place. In fact, it seems that $\theta$ remains uniformly bounded in $C^\frac{1}{2}$ for some unapparent reason. Proving that this is actually the case would be very interesting since it indicates that simply by choosing the velocity field $u$ to be $H(\theta)$ rather than $\theta$, a phenomenon which "breaks the scaling" of the equation takes place. In other words, the dynamics of the equation would then not be dictated by clearly conserved quantities or scaling, since the strongest known conserved quantity is the $W^{1,1}$ norm of $\theta$ and the scaling critical quantity is the $W^{1,\infty}$ norm. The $C^\frac{1}{2}$ conjecture is then simply to prove that solutions to the Cordoba-Cordoba-Fontelos model remain uniformly bounded in $C^\frac{1}{2}$ up to the time of blow-up (see \cite{Do},\cite{LR},\cite{SV},\cite{Kis},\cite{CCF}). We close by mentioning a recent work by Hoang and Radosz \cite{HR} which rigorously established cusp formation in a non-local model, which is a slight variant of the CCF model. 

\subsubsection{The models of Okamoto, Sakajo, and Wunsch}

The culmination of all of these developments is the introduction of one continuum of models under which one can place all of those previously stated. This was done by Okamoto, Sakajo, and Wunsch \cite{OSW}. They introduced the following continuum of models:

$$\partial_t \omega + a u\partial_x\omega=\omega\partial_x u,$$
$$\partial_x u=H(\omega),$$ with $a\in\mathbb{R}$ a parameter. The case $a=-1$ corresponds to the Cordoba-Cordoba-Fontelos model, the case $a=0$ corresponds to the Constantin-Lax-Majda model, and the case $a=1$ corresponds to De Gregorio's model. While it may seem that different values of $a$ can lead to totally different phenomena which would bring into question the efficacy of such a continuum, we will argue in this work that there seems to be a deeper connection between all of these models as well as the questions mentioned above: particularly the global regularity question for De Gregorio's model and the $C^\frac{1}{2}$ conjecture for the Cordoba-Cordoba-Fontelos model.  This connection is most visible when one considers self-similar solutions. 

Returning to the issue of finite time singularities, it is intuitively clear that when $a < 0$, the transport effect would work together with the stretching effect to develop a singularity. Indeed, Castro and C{\'o}rdoba \cite{CaCo} managed to prove that there is finite time blow up of smooth solutions in the case of negative $a$. The question of finite time singularity/global regularity was left open for $a > 0$, and one of our main results shows that when $a$ is small and $C^\infty$ data is considered, the blow up indeed occurs, and it may occur even in a self-similar way. Moreover, for each $a>0$ there exists initial data in $C^{\alpha(a)}$ for some $\alpha(a) > 0$ sufficiently small for $a$ large, which leads to a finite-time singularity. 

\subsubsection{Connections with the SQG Equation}

It is interesting to note that both the De Gregorio and Okamoto-Sakajo-Wunsch systems have direct connections with the SQG equation discussed in the above. Indeed, in the work of Castro and C{\'o}rdoba \cite{CaCo}, it was observed that with the so-called \textit{stagnation point similitude} ansatz \begin{equation*}
\begin{split}
\theta(t,x_1,x_2) = x_2 f(t,x_1), 
\end{split}
\end{equation*} if $f$ solves the De Gregorio system, then one obtains a solution to the SQG equation. Moreover, in our recent work \cite{EJ1}, we have established that one can uniquely solve the SQG equation (locally in time) with the \textit{radial homogeneity} ansatz, which is in polar coordinates \begin{equation*}
\begin{split}
\theta(t,r,s) = r \cdot h(t,s) ,\qquad  s \in [-\pi,\pi)
\end{split}
\end{equation*} and then $h(t,\cdot)$ on the circle solves a system which is essentially the Okamoto-Sakajo-Wunsch system with $a = 2$, up to a very smooth term in the relation connecting the scalar $h$ with the velocity $u$. The importance of this ansatz (as opposed to the stagnation point ansatz) is that singularity formation from smooth data for radially homogeneous solutions (which have infinite energy) implies finite-time singularity formation from compactly supported (finite energy) data in a suitable local well-posedness class, as shown in \cite{EJ1}.

%Closing this section, let us note that the case $a = -2 $ have appeared in the work of C{\'o}rdoba, C{\'o}rdoba, Fontelos \cite{CCF} as a model for the SQG equation, and also as a model of the vortex sheet problem in a paper of Baker, Li, and Morlet \cite{BLM}. Finite time blow up was established in \cite{CCF}.
%Based on this observation, one may speculate that once one decreases the coefficient of the transport term, there is a finite time singularity, and global regularity upon increasing the coefficient of the transport term. Along these lines, Okamoto-Sakajo-Wunsch \cite{OSW} introduced the generalized system \eqref{eq:DeG}. 

\subsection{The Setting of the Present Paper and the Main Results}

In this paper, we consider the Okamoto-Sakajo-Wunsch systems in the following form:\footnote{Note that the form of the equation is slightly different from that in the introduction and in the literature \cite{CaCo,CLM,DG1,OSW}, but it is equivalent up to a rescaling of time. Our convention makes some of the formulas a bit simpler.} 
\begin{equation}\label{eq:DeG}
\begin{split}
\partial_t \omega+ a\,u\partial_x \omega = 2 \omega\partial_x u, \qquad u = -\Lambda^{-1}\omega
\end{split}
\end{equation} where $a\in\mathbb{R}$ is a parameter and $u(t,\cdot)$ and $\omega(t,\cdot)$ are functions of one spatial variable $x \in \mathbb{R}$. Our first main result shows that when $|a|$ is small, there exists a smooth initial data which blows up in finite time in a self-similar manner. This is to say that 

\begin{theorem}\label{thm:main1}
	There exists some absolute constant $a_0 > 0$ such that for $|a| < a_0$, there is an odd initial data $\omega_0 = \omega_0^a \in H^3(\mathbb{R}) $, depending analytically on $a$, which blows up in finite time in a self-similar manner.
	
	More precisely, for $|a| < a_0$, there exists some $\lambda = \lambda(a)$ depending analytically on $a$ and with $\lambda(0) = 0$, such that the unique local solution of \eqref{eq:DeG} with initial data $\omega_0:=\omega_0^a$ in $C([0,1);H^3(\mathbb{R}))$ is given by \begin{equation*}
	\begin{split}
	\omega(t,x) = \frac{1}{1-t} \omega_0\left( \frac{x}{(1-t)^{1+\lambda}} \right),
	\end{split}
	\end{equation*} which blows up precisely at $t = 1$. 
\end{theorem}

\begin{remark}
	Even though we have chosen to work with a finite regularity space $H^3(\mathbb{R})$, a straightforward modification of our arguments actually gives that the profile $\omega_0$ belongs to any $H^s(\mathbb{R})$ for $s \ge 0$ and in particular to the space of analytic functions. 
\end{remark}

\begin{theorem}\label{thm:main1prime}
	Let $\omega^{a}(x,t)$ be the self-similar solutions constructed in Theorem \ref{thm:main1}. Then, $\omega^a$ remains uniformly smooth up to the blow-up time away from $x=0$. That is, given $b>0$, $$\sup_{(x,t)\in (\mathbb{R}\backslash (-b,b))\times [0,1]}|\omega^a(x,t)|\leq C(b).$$  
	Moreover,  $\omega^a$ remains in the space of weak $L^{1+\lambda(a)}$ functions uniformly up to the blow-up time. That is, $\omega^a\in L^\infty([0,1]; L^{1+\lambda(a),\infty}).$
	Finally, $$\sup_{t\in [0,1]}\int_{-1}^{1} |\omega^a(x,t)|^pdx<\infty$$ for all $p<1+\lambda(a)$. 
\end{theorem}

\begin{remark}
As a part of our construction, we shall see that $\lambda'(0)=-2+\ln 4<0$. Hence, we deduce that for small $a<0,$ $\omega$ belongs to $L^p([-1,1])$ for all $1\leq p<1+\lambda(a).$
\end{remark}

\begin{corollary}
Let $\omega^a$ be the self-similar solutions constructed in Theorem \ref{thm:main1} and suppose that $a<0$ is sufficiently small. Then, since $u^a=-\Lambda^{-1}\omega^a$, $$\sup_{t\in [0,1]}\V u^a\V_{C^{\alpha(a)}}<\infty$$ with $\alpha(a)=1-\frac{1}{1+\lambda(a)}.$
\end{corollary}

\begin{remark}
Notice that with our convention the De Gregorio model is the case $a=2$ and the Cordoba-Cordoba-Fontelos model is the case $a=-2$. Hence, to confirm the regularity conjecture for the De Gregorio model, one would need to show that $\lambda(2)=-2$ and to confirm the $C^\frac{1}{2}$ conjecture for the CCF model, one would need to show that $\lambda(-2)=1$ -- assuming that $\lambda$ can be continued to $|a|\leq 2$. 
\end{remark}
The second main theorem is one on finite-time singularity for \emph{all} the Okamoto-Sakajo-Wunsch models -- including De Gregorio's model -- by taking solutions with sufficiently low H\"{o}lder continuity index. 

\begin{theorem}\label{thm:main2}
	There exists some absolute constant $c_0 > 0$ such that for $\alpha \in  \{1/n, n \in \mathbb{N} \}$ and $|a| < c_0/\alpha$, there  exists some value $\lambda^{(\alpha)}(a) $ satisfying $\lambda^{(\alpha)}(0) = 0$, $\lambda^{(\alpha)}(a) > -2$, and %\begin{equation*}
	%\begin{split}
	%1 < p := \frac{1+\lambda^{(\alpha)}(a)}{\alpha}  < \infty,
	%\end{split}
	%\end{equation*} 
	an odd initial data $\omega_0^{(\alpha)}(a) \in L^{p,\infty}\cap C^\alpha(\mathbb{R})$ for $p=\frac{1+\lambda^{\alpha}(a)}{\alpha}$ which blows up in a self-similar way.
	
	More precisely, the unique local in time solution of \eqref{eq:DeG} with initial data $\omega_0$ in $C([0,T);C^\alpha(\mathbb{R}))$ is given by \begin{equation*}
	\begin{split}
	\omega(t,x) = \frac{1}{1-t} \omega^{(\alpha)}_0\left( \frac{x}{(1-t)^{\frac{1+\lambda^{(\alpha)}(a)}{\alpha}}    }\right),
	\end{split}
	\end{equation*} which blows up at $t = 1$.
\end{theorem}

\begin{corollary}
For $a=-2,$ the Cordoba-Cordoba-Fontelos model, there exists a family of solutions $\omega^\alpha(x,t)$ initially belonging to the class $L^{\frac{1+\lambda(\alpha)}{\alpha},\infty}\cap C^{\alpha}(\mathbb{R})$ for all $|\alpha|<\alpha_0$ a small constant and such that $\lambda(\alpha)\leq C_0 \alpha$ as $\alpha\rightarrow 0$ which blows up at $t=1$ in a neighborhood of $x=0$. Moreover, the associated velocity field $u^\alpha$ remains uniformly bounded in the space $C^{1-\frac{1}{1+\lambda(\alpha)}}([-1,1])$ up to the time of blow-up.  Notice that $\frac{1}{2}<1-\frac{1}{1+\lambda(\alpha)}$ for $\alpha$ sufficiently small.  
\end{corollary}
This is to say that for each $\epsilon>0$ there are solutions to the CCF model which become singular in a way that keeps the $C^{1-\epsilon}$ norm of the velocity under control up to the time of singularity where the $C^1$ norm of the velocity becomes infinite.

\subsection{A toy model}
The main idea of our work is that if the transport term is "weaker" than the vortex stretching term, then singularities will form in finite time. The weakening of the transport term, in our setting, can happen in two ways:

\begin{enumerate}
\item By putting a small parameter in front of the transport term. 
\item By taking merely $C^\alpha$ vorticity with $\alpha > 0$ small.\footnote{We would like to mention a relevant work of Zlato\v{s} \cite{Z} who obtained exponential in time growth of the vorticity gradient in $L^\infty$ for the 2D Euler equation on $\mathbb{T}^2$. This was done by using initial vorticity whose gradient at the origin is only $C^\alpha$-regular.}
\end{enumerate}

We illustrate these ideas using an explicitly solvable toy model. Consider the following model where one simply replaces $\partial_x u$ in \eqref{eq:DeG} by its value at zero $\partial_x u(0) = -H\omega(0)$: \begin{equation}\label{eq:toy}
\begin{split}
\partial_t \omega - a H\omega(0) x \cdot \partial_x \omega + 2H\omega(0) \cdot \omega = 0,
\end{split}
\end{equation} where we assume the initial data $\omega_0$ to be odd, nonnegative on $\mathbb{R}^+$, and smooth away from the origin. Then, solving along the characteristics, \begin{equation*}
\begin{split}
\omega(t,x) = e^{- 2\int_0^t H\omega(s,0) ds } \cdot \omega_0( x e^{a\int_0^t H\omega(s,0) ds }),
\end{split}
\end{equation*} and using this representation for $\omega(t,\cdot)$ we compute \begin{equation*}
\begin{split}
H\omega(t,0) = - \frac{2}{\pi} \int_0^\infty  e^{- 2\int_0^t H\omega(s,0) ds } \omega_0( x e^{a\int_0^t H\omega(s,0)}) \frac{dx}{x} = - e^{- 2\int_0^t H\omega(s,0) ds } c_0,
\end{split}
\end{equation*} where $c_0 = - H\omega_0(0) > 0$, unless $\omega_0$ is identically zero. This in turn determines \begin{equation*}
\begin{split}
H\omega(t,0) = -\frac{1}{c_0^{-1} - 2t} ,
\end{split}
\end{equation*} and inserting this back into the formula for $\omega(t,\cdot)$, we conclude \begin{equation*}
\begin{split}
\omega(t,x) = \frac{1}{1-2c_0 t} \omega_0\left(  \frac{x}{(1 - 2c_0 t)^{-a/2}} \right).
\end{split}
\end{equation*}

We observe that: \begin{itemize}
	\item For $a < 2$, the solution develops a singularity at time $t^* = (2c_0)^{-1}$, for any smooth profile $\omega_0$ with $\omega_0'(0)\not=0$. 
	\item In the case $a = 2$, if the initial data is $C^1$ or better, then the solution stays smooth up to the time moment $(2c_0)^{-1}$ and can be continued as a smooth solution for all time; indeed, using the mean value theorem with $\omega_0(0) = 0$, \begin{equation*}
	\begin{split}
	\left| \frac{\omega_0(x(1-2c_0 t))}{1-2c_0 t} \right| \le x \V \omega_0'\V_{L^\infty}.
	\end{split}
	\end{equation*} On the other hand, if initially \begin{equation*}
	\begin{split}
	\omega_0(x) \approx C\,\mathrm{sgn}(x) |x|^\alpha
	\end{split}
	\end{equation*} near the origin for some $0 < \alpha < 1$ and for some constant $C > 0$, then the solution blows up in $L^\infty$ at time $(2c_0)^{-1}$. 
\end{itemize}

\subsection{The Self-Similar Ansatz}

In the case of $a = 0$, the system \eqref{eq:DeG} is just the Constantin-Lax-Majda equation, which has a simple and explicit self-similar solution \begin{equation*}
\begin{split}
\omega(t,x) = \frac{1}{1-t} F_0\left( \frac{x}{1-t} \right), \qquad F_0(z) = \frac{z}{1+z^2},
\end{split}
\end{equation*} see Figure \ref{fig:CLM_evolution}. 
%In fact, we note that blow-up in the CLM equation is \emph{generically} asymptotically equal to this self-similar solution at the time of blow-up.
%\begin{proposition}
%Let $\omega_0\in C^\infty(\mathbb{R})\cap L^2(\mathbb{R})$ be odd and positive on $(0,\infty)$ and such that $\omega_0'(0)\not=0$ and $H(\omega_0)(0)\not=0$. 
%\end{proposition}

The scaling symmetry of the model permits one to work in the self-similar variable $z = x/(1-t)^{1+\lambda}$ for any $\lambda \in \mathbb{R}$, and it is easy to see that a profile $F$ gives a self-similar solution to \eqref{eq:DeG} if and only if $F$ solves the differential equation  \begin{equation*}
\begin{split}
F(z) + \left( (1+\lambda)z - a \Lambda^{-1}F(z) \right) F'(z) + 2F(z) HF(z) = 0. 
\end{split}
\end{equation*} Then we proceed to show existence of a solution by essentially  ``linearizing'' around the special solution which corresponds to $a = 0, \lambda = 0$, and $F = F_0$. Proving the main theorem boils down to establishing that the inverse of the associated linear operator is bounded. 

\subsubsection{The $a=0$ case.}

In the special case of $a = 0$, the system \eqref{eq:DeG} reduces to the Constantin-Lax-Majda equation \cite{CLM} on $\mathbb{R}$: \begin{equation}\label{eq:CLM}
\begin{split}
\pr_t \omega + 2H(\omega)  \omega  = 0. 
\end{split}
\end{equation} 
%In the original work of Constantin, Lax, and Majda \cite{CLM}, the equation \eqref{eq:CLM} was suggested as a model for the vortex dynamics in the 3D Euler equations.

A particularly nice feature of this model is that it can be integrated: one may check that the pair \begin{equation}\label{eq:CLM_formula}
\begin{split}
\omega(t,x) = \frac{\omega_0(x)}{(1+t \cdot H\omega_0(x))^2 + (t\cdot \omega_0(x) )^2}, \quad H\omega(t,x) = \frac{(H\omega_0)(x)(1 + t\cdot H\omega_0(x)) + t\cdot \omega_0^2(x)}{(1+t \cdot H\omega_0(x))^2 + (t \cdot \omega_0(x) )^2}
\end{split}
\end{equation} provides a solution to \eqref{eq:CLM}. Assuming that the set \begin{equation*}
\begin{split}
Z = \{ x : \omega_0(x)  = 0 , \, H\omega_0(x) < 0 \}
\end{split}
\end{equation*} is nonempty, the formula \eqref{eq:CLM_formula} explicitly shows that $\omega(t,\cdot)$ blows up in $L^\infty(\mathbb{R})$ at $$t^* = ( \sup_{x \in Z} -H\omega_0(x)    )^{-1} > 0, $$ at the point of $Z$ which attains the supremum. In particular, when the initial data $\omega_0(x)$ is an odd function of $x$ and strictly positive on $(0,\infty)$, the solution becomes singular at the origin exactly at $t^* = -H\omega_0(0)^{-1} > 0$.

\begin{figure}
	\includegraphics[scale=0.4]{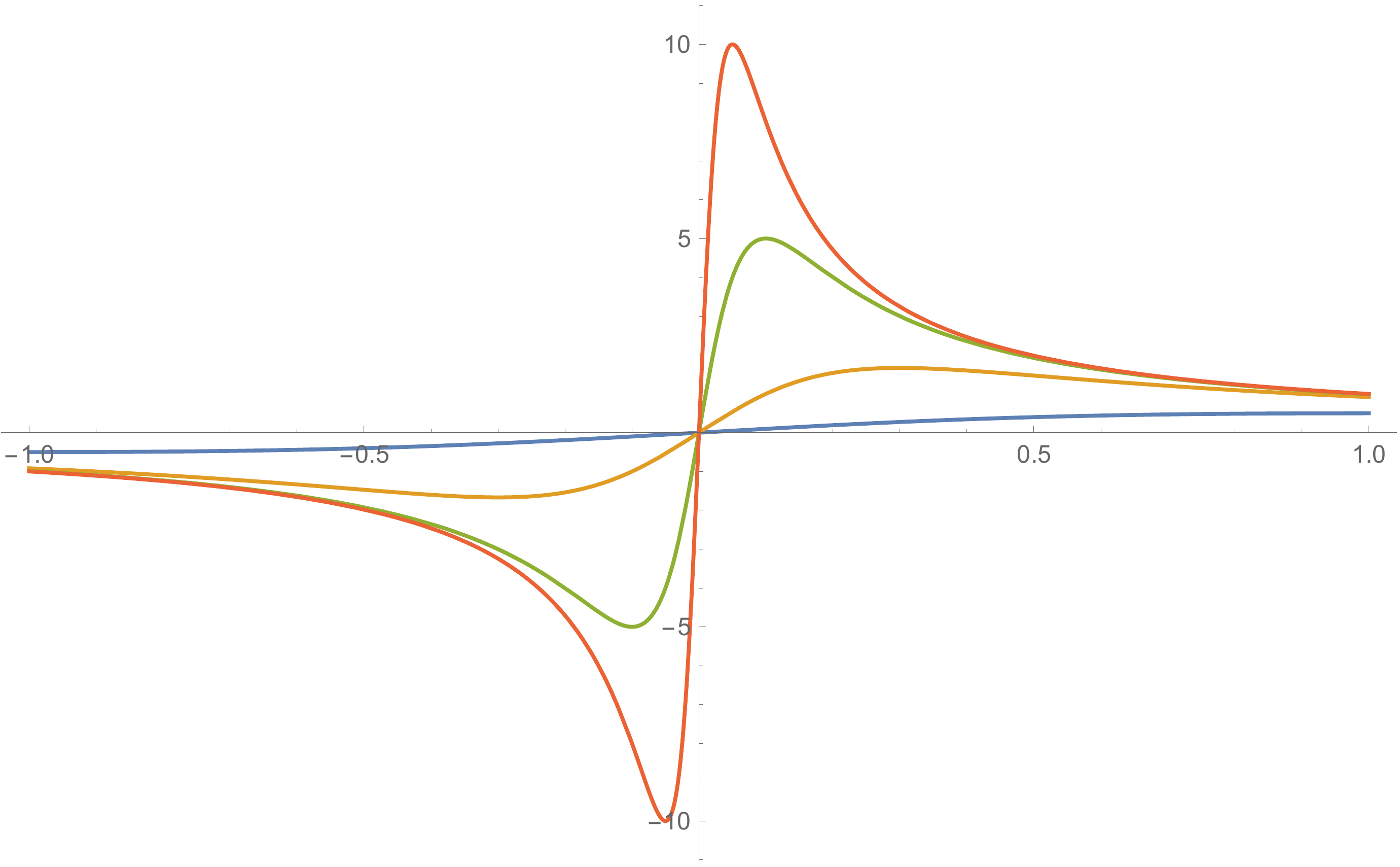} 
	\centering
	\caption{Evolution of the data $\omega_0(x) = x/(1+x^2)$ at time moments $t = 0.95, 0.9, 0.7, 0$ (from top to bottom).}
	\label{fig:CLM_evolution}
\end{figure}

As a special case, consider the initial data \begin{equation*}
\begin{split}
\omega_0(x) = \frac{x}{1+x^2}, 
\end{split}
\end{equation*} whose Hilbert transform equals \begin{equation*}
\begin{split}
H\omega_0(x) = -\frac{1}{1+x^2}.
\end{split}
\end{equation*} Plugging these identities into \eqref{eq:CLM_formula}, one can see that \begin{equation}\label{eq:CLM_special}
\begin{split}
\omega(t,x) = \frac{\frac{x}{x^2+1}}{\left( 1 - \frac{t}{x^2 + 1} \right)^2 + t^2 \cdot  \left(\frac{x}{x^2 + 1}\right)^2} = \frac{1}{1-t} \cdot \frac{\frac{x}{1-t}}{1+ \left(\frac{x}{1-t}\right)^2} = \frac{1}{1-t} \cdot \omega_0\left(\frac{x}{1-t}\right),
\end{split}
\end{equation} that is, the solution is self-similar and blows up at time $t = 1$; see Figure \ref{fig:CLM_evolution}. Note that the ``vorticity stretching'' factor $H\omega(t,\cdot)$ is always maximal at the origin, and therefore it drags the maximum point of $\omega(t,\cdot)$ towards the origin, and at the same time gets intensified.

\subsubsection{The case $a>0$.}

Moving on to the case of $a > 0$, let us stick to the situation where $\omega(t,\cdot)$ is odd and positive on $(0,\infty)$. Then, the velocity $u(t,\cdot)$ is again odd and always directed away from the origin, and therefore it counteracts the stretching term which tries to form a sharp gradient by transporting the maximum point towards the origin.
It is far from clear whether the stretching term can overcome the transport effect and still result in a finite-time blow up; what seems to happen is that for some critical value of $a$ (likely $a=2$), the transport term takes over and prevents the singularity formation. We conjecture that for smooth data there will be global regularity when $a\geq 2$ and singularity formation when $a<2$. Theorem \ref{thm:main1} confirms this for all $a<a_0$ for some constant $a_0>0$. For merely $C^\alpha$ data we will show that there can be singularity formation for any value of $a$ so long as $\alpha$ is taken small enough and this is the content of Theorem \ref{thm:main2}.

\subsection{The main steps in the proofs of Theorems \ref{thm:main1} and \ref{thm:main2}}

We close the introduction by giving a succinct explanation of how we go about proving Theorems \ref{thm:main1} and \ref{thm:main2}. 
\begin{enumerate}
\item \emph{Construction of self-similar solutions when $a=0$ which become singular in finite time.} To do this, one studies carefully the solution formula for the solutions to the $a=0$ equation. In the smooth case, one uses Taylor's theorem to extract a universal self-similar profile. In the $C^\alpha$ case, it turns out that we can do a similar procedure using an interesting fact about the Hilbert transform (Lemma \ref{lem:consistency_alpha}). This lemma says that if $g(z)=\mathrm{sgn}(z)k(|z|^\alpha)$ for some $0<\alpha<1$ with $k\in C^\infty$ and $k(0)=0$ then $\lim_{z\rightarrow 0}\frac{H(g)(z)-H(g)(0)}{|z|^\alpha}$ is well defined and is determined \emph{only} by $k'(0)$. One can then check that the asymptotic profiles derived this way are actually self-similar solutions. This leaves us with one smooth self similar solution when $a=0$ as well as a family of merely $C^\alpha$ self-similar solutions. In the smooth case, this solution is just of the form $\omega(x,t)=\frac{1}{1-t}F(\frac{x}{1-t})$ with $F(z)=\frac{z}{1+z^2}.$

\item \emph{Analysis of the linearization of the equation satisfied by self-similar solutions around the $a=0$ solution.} Normally, when one considers a functional equation $N(f,a)=0$ along with a known solution $(f^0,0)$, proving the existence of a family of solutions $(f^a,a)$ for $a$ small can be done using the Implicit Function Theorem applied in the right functional setting. Applying the Implicit Function Theorem, in turn, requires invertibility of the linear operator $\frac{\partial N}{\partial f}(f_0,0).$ Unfortunately, this linear operator is neither onto nor one-to-one in our case. In fact, for an $H^3(\mathbb{R})$ function to be in the range of this linear operator, it must satisfy a consistency condition relating its Hilbert transform at $0$ and its derivative at 0. However, since this consistency condition is one-dimensional, this problem can be remedied by introducing one other parameter which we call $\lambda(a)$. In particular, instead of looking merely for a solution of the form $\frac{1}{1-t}F^a(\frac{x}{1-t}),$ we now look for a solution of the form $\frac{1}{1-t}F^a(\frac{x}{(1-t)^{\lambda(a)}}).$ A judicious choice of $\lambda(a)$ allows us to carry out an iteration procedure to formally continue the $a=0$ solution and the iteration procedure converges so long as certain operators can be bounded on the spaces we are working in. These bounds become very delicate in the $C^\alpha$ case, especially since we need to be careful to get bounds with explicit dependence on $\alpha$. 

\item \emph{Technical lemmas related to the Hilbert transform and related operators.} To prove boundedness of the inverse of the modified linearized operator discussed above, we must use several important technical lemmas, some of which are classical and some of which seem to be new. We often make use of the Tricomi identity \eqref{eq:tricomi} and the classical Hardy inequality \eqref{eq:Hardy}. In fact, these are all we need in the smooth case. In the $C^\alpha$ case we need to prove (sharp) bounds for Hilbert transform-like operators which describe how the Hilbert transform acts on functions of $|x|^\frac{1}{n}$ for $n\in\mathbb{N}$. If for some $n\in\mathbb{N},$ $f(x)=\tilde f(|x|^\frac{1}{n})$ with $\tilde f$ an even $L^2$ function then we define the operators\footnote{For clarity of exposition in the proof, we will often write these operators simply as $\tilde H$ since $n$ is always kept fixed.} $\tilde H^{(n)}$ by $\tilde H^{(n)}(\tilde f)(w):= \mathrm{sgn}(w)H(f)(|w|^n).$ In particular, it can be shown that the operators $\tilde H^{(n)}$ are bounded on the space of $L^2$ even functions with operator norms $C n$ for some universal constant $C>0$ and all $n\in\mathbb{N}$.  We also need the Hardy-type inequalities \eqref{eq:Hardy_alpha}.
\end{enumerate}

\subsection{Organization of the paper}
In Section \ref{sec:CLM} we discuss self-similar blow-up solutions for the CLM equation and how blow-up for CLM is generically self-similar. We discuss two cases therein: smooth self-similar profiles and merely H\"older continuous self similar profiles.  In Section \ref{sec:smooth}, we prove that the self-similar profiles for the CLM equation in the smooth case can be continued to self-similar blow-up profiles for the Okamoto-Sakajo-Wunsch equation with $|a|$ small which is the content of Theorem \ref{thm:main1}.  In Section \ref{sec:Hoelder}, we do the same for the merely H\"older continuous profiles for the CLM equation but there $|a|$ can be taken to be on the order of $\alpha^{-1}$ where $\alpha$ is the H\"older continuity index of the CLM solution, which is the content of Theorem \ref{thm:main2}. Lastly, we discuss an  approach which could lead to showing finite-time singularity formation for the De Gregorio system on the circle in Section \ref{sec:circle}. Several useful identities regarding the Hilbert transform and some functional inequalities are collected in the Appendix.

\subsubsection*{Notation} As it is usual, we use letters $C, c$, and so on to denote absolute constants whose values may vary from a line to another. Moreover, we write $X \lesssim Y$ when there is some absolute constant $c > 0$ with $X \le cY$. When a constant or a function depends on a parameter, we write out the dependence using either a subscript or a superscript. In particular, for $0 < \alpha < 1$, the expression $F^{(\alpha)}_0$ is reserved for the odd and self-similar profile for the Constantin-Lax-Majda equation; see \eqref{eq:F_alpha}. In Section \ref{sec:Hoelder}, we use tildes to denote the transform which acts on functions on $\mathbb{R}^+$ defined via \begin{equation*}
\begin{split}
\tilde{f}(w) := f(w^{\frac{1}{\alpha}})
\end{split}
\end{equation*} with some given value of $0 < \alpha < 1$. 

\section{Self-similar solutions for Constantin-Lax-Majda: Two Cases}\label{sec:CLM}

In this section we completely characterize the blow-up profile for solutions to the Constantin-Lax-Majda equation which are odd and positive on $(0,\infty)$. It shouldn't be too surprising that we are able to get a full picture of blow-up phenomena for the CLM equation since it can be solved exactly. It turns out that the nature of the blow-up depends only upon the behavior of $\omega_0$ near $x=0$. Moreover, the blow-up is of self-similar type with different scaling rates depending upon whether $\omega_0'(0)=0,$ $0<\omega_0'(0)<\infty,$ or $\omega_0'(0)=+\infty$. By "of self-similar type" we mean that the solution consists of purely self similar function plus a function which is uniformly bounded up to the time of blow-up. We now go into more detail on only the second and third cases since we shall not need the third case. Recall the solution formula for the Constantin-Lax-Majda equation:
$$\omega(x,t)= \frac{\omega_0(x)}{(1+tH\omega_0(x))^2+t^2\omega_0(x)^2}.$$

\subsection{Case 1: $0<\omega_0'(0)<\infty$}

Let's first consider the case where $\omega_0'(0)=1$. Also assume\footnote{We remark that $H\omega_0(0)=-\frac{1}{\pi}\int_{-\infty}^\infty\frac{\omega_0(x)}{x}dx$ so that $H\omega_0(0)\leq 0$ and there is equality if and only if $\omega_0\equiv 0$. Using scaling and multiplication by a constant, it actually suffices to assume both $\omega_0'(0)=1$ and $H\omega_0(0)=-1$.} that $H\omega_0(0)=-1$. Then, $$\omega(x,t)=  \frac{x}{(1+tH\omega_0(x))^2+t^2\omega_0(x)^2}+\frac{\omega_0(x)-x}{(1+tH\omega_0(x))^2+t^2\omega_0(x)^2}.$$
Notice that the second term is uniformly bounded for $\frac{1}{2}\leq t\leq 1$ since $$|\omega_0(x)-x|\lesssim |x|^3$$ and $$|\omega_0(x)|^2\geq \frac{1}{2}|x|^2.$$
Hence, $$\omega(x,t)\approx  \frac{x}{(1+tH\omega_0(x))^2+t^2\omega_0(x)^2},$$ where the $\approx$ sign means equal up to a term uniformly bounded up to $t=1$. Now, $$|H\omega_0(x)+1|\lesssim |x|^2$$ since $H\omega_0$ is even. Hence, 
$$\omega(x,t)\approx\frac{x}{(1-t)^2+t^2\omega_0(x)^2}+ \Big[ \frac{x}{(1+tH\omega_0(x))^2+t^2\omega_0(x)^2}-\frac{x}{(1-t)^2+t^2\omega_0(x)^2}\Big]$$
$$ \frac{x}{(1+tH\omega_0(x))^2+t^2\omega_0(x)^2}-\frac{x}{(1-t)^2+t^2\omega_0(x)^2}=\frac{x [(1-t)^2-(1+tH\omega_0(x))^2]}{[(1-t)^2+t^2\omega_0(x)^2][(1+tH\omega_0(x))^2+t^2\omega_0(x)^2]}$$
Now, $$ |(1-t)^2-(1+tH\omega_0(x))^2|=|t(1+H\omega_0(x))(2-t+tH\omega_0(x))|\lesssim (1-t)x^2+|x|^4$$
Therefore, $$\omega(x,t)\approx\frac{x}{(1-t)^2+t^2\omega_0(x)^2}.$$
Finally, $$\omega(x,t)\approx\frac{x}{(1-t)^2+t^2x^2}=\frac{1}{1-t}F(\frac{x}{1-t})$$ with $F(z)=\frac{z}{z^2+1}.$
As mentioned in the footnote, a simple scaling argument will then imply that whenever $0<\omega_0'(0)<\infty,$ and $\omega_0\in C^2,$ $\omega$ is asymptotically equal to a self similar profile. In fact, this is the only such self-similar profile.

It is straightforward to check that with the self-similar ansatz \begin{equation*}
\begin{split}
\omega(t,x) = \frac{1}{1-t} F_0\left(\frac{x}{1-t}\right),
\end{split}
\end{equation*} the Constantin-Lax-Madja equation \eqref{eq:CLM} reduces to the following: 
\begin{equation}\label{eq:CLM_resc}
\begin{split}
F_0+zF_0'(z)+2F_0H(F_0) = 0. 
\end{split}
\end{equation} We have already seen that $F_0(z) = Cz/(1+C^2z^2)$ is a solution to this equation. We also have:
\begin{proposition}
	Assuming that the profile $F_0(z)$ in \eqref{eq:CLM_resc} is odd, smooth, and decays at infinity, we have \begin{equation*}
	\begin{split}
	F_0(z) = \frac{z}{1+z^2},
	\end{split}
	\end{equation*} up to a scaling of the variable $z$. 
\end{proposition}

\begin{proof}
	Taking the Hilbert transform of both sides of the equation \eqref{eq:CLM_resc}, we see:
	\begin{equation}\label{eq:CLM_resc_H}
	\begin{split}
	H(F_0)+zH(F_0)' - F_0^2  + H(F_0)^2 = 0 .
	\end{split}
	\end{equation}
	Now let $F_0+iH(F_0)=V.$
	Then, since
	$$2F_0H(F_0)+i(H(F_0)^2-F^2)=-i(F_0+iH(F_0))^2=-iV^2,$$ the system of equations \eqref{eq:CLM_resc}, \eqref{eq:CLM_resc_H} can be rewritten as $$V+zV'- iV^2 = 0 .$$ Integrating the system gives $$V=\frac{1}{i+Cz},$$ and it is easy to see that the real and imaginary parts of $V$ are odd and even respectively precisely when $C$ is a real constant. 
	Thus, $$V=\frac{Cz}{1+C^2z^2}+\frac{i}{1+C^2z^2}$$ and in particular, $$F_0=\frac{Cz}{1+C^2z^2}.$$
\end{proof}

%\subsection{Case 2: $\omega_0'(0)=0$}

%A first case to consider the case\footnote{If we assume that $\omega_0$ is analytic, it suffices to consider $\omega_0^{(n)}(0)=1$ for some odd $n\in\mathbb{N}$ and $\omega_0^{(k)}(0)=0$ for all $k<n$. Since we won't use this case very much in the rest of the paper, we will restrict ourselves to the case $\omega_0'''(0)=6$.} where $\omega_0'(0)=0$ and $\omega_0''(0)=2$. In this case, $$\omega(x,t)= \frac{\omega_0(x)}{(1+tH\omega_0(x))^2+t^2\omega_0(x)^2}.$$
%We know that $|\omega_0(x)-x^3|\lesssim x^5.$

\subsection{Case 2: $\omega_0'(0)=+\infty$}

Let's assume that $\omega_0(x)=\mathrm{sgn}(x)|x|^\alpha \Omega_1(x)$ for some $0<\alpha<1$ and for some $\Omega_1$ a sufficiently smooth rapidly decaying even function with $\Omega_1(0)=1$ and $\Omega_1>0$ on $\mathbb{R}$. In this case, it is clear that $H\omega_0(x)=C+ |x|^\alpha\Omega_2(x)$ with $C$ some constant and $\Omega_2(x)$ a smooth even decaying function on $\mathbb{R}.$ We actually need the following interesting lemma: 
\begin{lemma}\label{lem:consistency_alpha}
Suppose $g(z)=\mathrm{sgn}(z) k(|z|^\alpha)$ with $k$ smooth and $k(0)=0$. Then, $$\lim_{z\rightarrow 0}\frac{1}{|z|^\alpha}[H(g)(z)-H(g)(0)]=\cot\left(\frac{\alpha\pi}{2}\right) k'(0).$$
\end{lemma}

\begin{proof} Assume $0<x<0.1$. Note that 
\begin{equation*}
\begin{split}
&\int_{-\infty}^{\infty} \frac{\mathrm{sgn}(z)k(|z|^\alpha)}{x-z}dz-\int_{-\infty}^{\infty} \frac{\mathrm{sgn}(z)k(|z|^\alpha)}{z}dz  = x\int_{-\infty}^\infty \frac{\mathrm{sgn}(z)k(|z|^\alpha)}{(x-z)z}dz.
\end{split}
\end{equation*} Since we are only concerned about the limit as $z\rightarrow 0$ of this quantity divided by $x^\alpha$, we can cut out the large $z$ part. So we see that the quantity we want is equal to $$\lim_{x\rightarrow 0}\frac{1}{x^\alpha}x\int_{-1}^{1} \frac{\mathrm{sgn}(z)|z|^\alpha}{(x-z)z}dz.$$ After a rescaling, it is equal to $$\lim_{x\rightarrow 0}p.v.\int_{-1/x}^{1/x}\frac{dz}{(1-z)|z|^{1-\alpha}}=p.v.\int_{-\infty}^\infty\frac{1}{(1-z)|z|^{1-\alpha}}dz. $$
A direct numerical check gives us that $$p.v.\int_{-\infty}^\infty \frac{1}{(1-z)|z|^{1-\alpha}}dz=\pi\cot\left(\frac{\alpha\pi}{2}\right),$$ and therefore multiplying by $\frac{1}{\pi}$, we are done. 
\end{proof}
Using the lemma, we see that $\Omega_2(0)=\cot(\frac{\alpha\pi}{2}).$
It is clear from the formula
$$\omega(x,t)= \frac{\omega_0(x)}{(1+tH\omega_0(x))^2+t^2\omega_0(x)^2}$$
that singularity formation happens at $x=0$ and $t=-C^{-1}.$ Note that if $\Omega_1$ is fixed, $C\approx -\frac{1}{\alpha}$ as $\alpha\rightarrow 0$. This can be seen by direct computation or by observing that $H(\mathrm{sgn}(x))$ has a logarithmic singularity at the origin. 
Using arguments similar to the above, we see:

$$\omega(x,t)\approx \frac{\mathrm{sgn}(x)|x|^\alpha\Omega_1(0)}{(1-C^{-1}(C+|x|^\alpha\Omega_2(0)))^2+C^{-2}|x|^{2\alpha}\Omega_1(0)^2}.$$
Note that $\Omega_1(0)=1$ and $\Omega_2(0)=\cot(\frac{\alpha\pi}{2})$. 
$$\omega(x,t)\approx \frac{\mathrm{sgn}(x)|x|^\alpha}{(1+t(C+|x|^\alpha \cot(\frac{\alpha\pi}{2})))^2+t^2|x|^{2\alpha}}.$$
$$=\frac{\mathrm{sgn}(x)|x|^\alpha}{(1+tC)^2+2(1+tC)|x|^\alpha\cot(\frac{\alpha\pi}{2})+t^2|x|^{2\alpha}\csc^2(\frac{\alpha\pi}{2})}$$
To make things easier, let's assume $C=-1$. 
So, $$\omega(x,t)\approx \frac{\mathrm{sgn}(x)|x|^\alpha}{(1-t)^2+2(1-t)|x|^\alpha\cot(\frac{\alpha\pi}{2})+|x|^{2\alpha}\csc^2(\frac{\alpha\pi}{2})}$$
Hence, we see that the behavior near the time of singularity $(t=1)$ is self similar:

$$\omega(x,t)\approx \frac{1}{1-t} F(\frac{x}{(1-t)^\frac{1}{\alpha}})$$ with $$F(z)=\frac{\mathrm{sgn}(z)|z|^\alpha}{1+2|z|^\alpha\cot(\frac{\alpha\pi}{2})+|z|^{2\alpha}\csc^2(\frac{\alpha\pi}{2})}.$$ This coincides with the exact self-similar solution \eqref{eq:F_alpha} up to a multiplicative constant and a term of order $|z|^{2\alpha}$. 

\section{Self-similar blow-up for smooth data}\label{sec:smooth}

In this section we prove Theorem \ref{thm:main1}. In particular, we consider the Okamoto-Sakajo-Wunsch equation and prove the existence of self-similar solutions which become singular in finite time for $|a|$ small enough. This is done using a continuation argument since we already have such a solution in the case when $a=0.$ This cannot actually be done directly since normally this requires invertibility of a certain linear operator and we will find the operator to be neither one-to-one nor onto. This will be fixed by introducing another parameter $\lambda(a)$ which changes the scaling of the self-similar solution.  Once this $\lambda(a)$ is chosen properly, the relevant linear operator becomes onto and once we make a further restriction the linear operator becomes invertible. Introducing this parameter $\lambda(a)$ is likely \emph{necessary} to construct smooth solutions of the Okamoto-Sakajo-Wunsch equation because, as is alluded to in the corollaries to Theorem \ref{thm:main1}, this seems to be related to the $C^{\frac{1}{2}}$ conjecture for the case $a=-2$. Namely, it seems to be that the velocity fields associated to solutions of the Okamoto-Sakajo-Wunsch equation with $a<0$ satisfy a uniform $C^{\alpha(a)}$ estimate for some $\alpha(a)>0$.  

\subsection{Self-similar solutions for small values of $a$}

This time, we consider the following ansatz \begin{equation*}
\begin{split}
\omega(t,x) = \frac{1}{1-t} F\left( \frac{x}{(1-t)^{1+\lambda(a)}} \right),\qquad \lambda(0) = 0
\end{split}
\end{equation*} for the system \eqref{eq:DeG}, where $F$ is an odd and smooth function which decays at infinity. Then, in terms of $F$, \eqref{eq:DeG} takes the following form: \begin{equation}\label{eq:DeG_resc}
\begin{split}
F(z) + \left( (1+\lambda(a))z - a \Lambda^{-1}(F)(z) \right) F'(z) + 2F(z) HF(z) = 0. 
\end{split}
\end{equation} Our previous computations show that when $a = 0$, then $F = F_0$ (with $\lambda(0) = 0$) solves \eqref{eq:DeG_resc}. 

We consider the following expansions in $a$ around $(F_0,0)$: \begin{equation}\label{eq:expansion}
\begin{split}
F(z) &= F_0(z) + \sum_{n=1}^\infty a^n F_n(z), \qquad \lambda(a) = \sum_{n=1}^\infty a^n \lambda_n
\end{split}
\end{equation} where $F_n$ and $\lambda_n$ for each $n \ge 1$ are to be determined via a simple iteration scheme. 

It turns out that  $F_n$ are $C^\infty(\mathbb{R})$ functions with suitable decay at infinity. For simplicity, we will just work in the space $H^3(\mathbb{R})$:\begin{equation*}
\begin{split}
\V f \V_{H^3(\mathbb{R})}^2 := \sum_{k=0}^3 \V \partial_z^k f\V_{L^2(\mathbb{R})}^2,
\end{split}
\end{equation*} and show that the series \eqref{eq:expansion} converges in $H^3(\mathbb{R})$. Note that for $f \in H^3$, the expressions such as $f'(0)$, $f''(0)$, and $Hf'(0)$ make sense, and in particular $F$ will solve \eqref{eq:DeG_resc} pointwise, once we show that the series for $F$ converges in $H^3$. 

Let us briefly outline how $F_n$ and $\lambda_n$ are going to be determined, given $F_1, \cdots, F_{n-1}$ and $\lambda_1,\cdots,\lambda_{n-1}$. We shall work with a linear operator $L$ defined on the space of odd functions, which has a one-dimensional kernel and the  image of codimension 1. Then at each step of the iteration, we will be able to determine $F_n$ by inverting $L$ in an equation of the form \begin{equation*}
\begin{split}
L(F_n)(z) = G_n(z) - \lambda_n z F_0'(z),
\end{split}
\end{equation*} with some $G_n  = G_n(F_1,\cdots,F_{n-1};\lambda_1,\cdots,\lambda_{n-1})$ where $\lambda_n$ is the unique number which gives a solution $F_n$, up to a one-dimensional kernel. Any non-trivial element in the kernel has a nonzero derivative at the origin, and therefore we may define $L^{-1}$ on the image of $L$ by forcing the derivative of $F_n$ to be zero at the origin. Then we proceed to show that, using simple norm estimates, both series in \eqref{eq:expansion} are convergent. 

\subsubsection{The linear operator and its inverse}

Let us insert the expansion \eqref{eq:expansion} directly into \eqref{eq:DeG_resc}. Then one obtains the sequence of identities, for each power of $a$: \begin{equation}\label{eq:seq_equations}
\begin{split}
F_n + zF_n' +2HF_0 F_n + 2F_0 HF_n = \sum_{j=0}^{n-1} \Lambda^{-1}({F_j}) F'_{n-1-j} - z\sum_{j=1}^n   \lambda_j F'_{n-j} - 2\sum_{j=1}^{n-1} F_j HF_{n-j},
\end{split}
\end{equation} for each $n \ge 1$. We note that, by considering the linear operator $L$ defined by \begin{equation}\label{eq:Lin_sys}
\begin{split}
Lf := f + zf' + 2HF_0 f + 2F_0 Hf = f + zf' - \frac{2}{1+z^2} f + \frac{2z}{1+z^2} Hf , 
\end{split}
\end{equation} the above set of equations can be rewritten as \begin{equation}\label{eq:n_th}
\begin{split}
L(F_n) = G_n(z) - \lambda_n  zF_0'(z) = G_n(z) - \lambda_n \frac{z(1-z^2)}{(1+z^2)^2}
\end{split}
\end{equation} where $G_n = G_n(F_1,\cdots,F_{n-1};\lambda_1,\cdots,\lambda_{n-1})$ is simply the right hand side of \eqref{eq:seq_equations} without the term involving $\lambda_n$. The following lemma will guarantee that we can always pick $\lambda_n$ in a (unique) way that \eqref{eq:n_th} is solvable. Then, we can define \begin{equation*}
\begin{split}
F_n := L^{-1}\left(  G_n(z) - \lambda_n \frac{z(1-z^2)}{(1+z^2)^2} \right)
\end{split}
\end{equation*}	by an explicit integral formula; see below \eqref{eq:inverse_rep}. 

The purpose of the following lemma is mainly to extract a representation formula for the $L^{-1}$, supposing that the right hand side is sufficiently smooth and decays  fast at infinity. Then in the next lemma, we carry out precise norm estimates for $F_n$. 

\begin{lemma}
	Consider the linear equation \begin{equation*}
	\begin{split}
	Lf = g
	\end{split}
	\end{equation*} where $g \in C^\infty_c(\mathbb{R})$ is a given odd function. Then a solution $f \in H^3(\mathbb{R})$ (which is necessarily odd) exists if and only if $g'(0) + 2Hg(0) = 0$ and once we further require that $f'(0) = 0$, it is uniquely determined by the expression \begin{equation}\label{eq:inverse_rep}
	\begin{split}
	(L^{-1}g)(z) &= \frac{z(1-z^2)}{(1+z^2)^2} \cdot \int_0^z \frac{1-w^2}{w} \hat{g}(w) + 2\hat{h}(w) dw \\
	&\qquad + \frac{-2z^2}{(1+z^2)^2} \cdot \int_0^z -2\hat{g}(w) + \frac{1-w^2}{w} \hat{h}(w)dw, 
	\end{split}
	\end{equation} with \begin{equation*}
	\begin{split}
	\hat{g}(z)&:= \frac{g(z)}{z} - \frac{1}{1+z^2}g'(0), \qquad	\hat{h}(z):= \frac{Hg(z) - Hg(0)}{z} + \frac{2z}{1+z^2}Hg(0).
	\end{split}
	\end{equation*}
\end{lemma}

\begin{proof}
	Assuming that we are given a solution $f \in H^3(\mathbb{R})$ to $Lf = g$, one may take the Hilbert transform of both sides to obtain \begin{equation}\label{eq:Lin_sys_H}
	\begin{split}
	Hf + z (Hf)' - \frac{2z}{1+z^2} f - \frac{2}{1+z^2} Hf = Hg,
	\end{split}
	\end{equation} where we have used the Tricomi identity. Setting $d = Hf$ for simplicity, we observe that a solution $f$ of $Lf = g$ must solve the following linear system of ODEs \begin{equation}\label{eq:Lin_sys_vectorform}
	\begin{split}
	\begin{pmatrix}
	f \\ d
	\end{pmatrix} + z \begin{pmatrix}
	f \\ d
	\end{pmatrix}' + \frac{2}{1+z^2} \begin{pmatrix}
	-1 & z \\
	-z & -1
	\end{pmatrix} \begin{pmatrix}
	f \\ d
	\end{pmatrix}  = \begin{pmatrix}
	g \\ h
	\end{pmatrix},
	\end{split}
	\end{equation} with $d = Hf$ and $h = Hg$. Since we are interested in the case where $f$ is odd (and therefore $d$ is necessarily even), it suffices to solve the initial value problem with given $f(0)$ and $d(0)$ on the positive reals. Due to the presence of the coefficient $z$, it is not clear whether the initial value problem is well-posed, and the condition $g'(0) + 2h(0) = 0$ comes out naturally in this setting. Indeed, evaluating both sides of the second component of \eqref{eq:Lin_sys_vectorform} at $z = 0$, one obtains \begin{equation*}
	\begin{split}
	d(0) - 2 d(0) = h(0),
	\end{split}
	\end{equation*} while dividing  both sides of the first component of \eqref{eq:Lin_sys_vectorform} by $z$ and taking the limit $z \rightarrow 0^+$, \begin{equation*}
	\begin{split}
	0 = \lim_{z \rightarrow 0^+} \left(	\frac{f(z)}{z} + f'(z) - \frac{2}{1+z^2} \frac{f(z)}{z} +\frac{2}{1+z^2} d(z) - \frac{g(z)}{z} \right) \rightarrow 2d(0) - g'(0). 
	\end{split}
	\end{equation*} Therefore we need \begin{equation*}
	\begin{split}
	d(0) = - h(0) = \frac{1}{2} g'(0)
	\end{split}
	\end{equation*} for the ODE \eqref{eq:Lin_sys_vectorform} to have a $C^1$ solution on $[0,\infty)$. In particular we see that the system \eqref{eq:Lin_sys_vectorform} selects a unique pair of initial conditions $(f(0),d(0)) = (0,-h(0)) = (0,g'(0)/2)$ which has a chance of having a smooth solution. 
	
	Introducing for convenience $\hat{d}(z) := d(z) - d(0)$ and using the above condition, one can simply rewrite the system as \begin{equation}\label{eq:Linsys_rewrite}
	\begin{split}
	\begin{pmatrix}
	f \\ \hat{d}
	\end{pmatrix}' + \frac{1}{z(1+z^2)} \begin{pmatrix}
	-(1-z^2) & 2z \\
	-2z & -(1-z^2) 
	\end{pmatrix} \begin{pmatrix}
	f \\ \hat{d}
	\end{pmatrix} = \begin{pmatrix}
	\frac{g(z)}{z} - \frac{1}{1+z^2}g'(0) \\ \frac{h(z)-h(0)}{z} + \frac{2z}{1+z^2} h(0)
	\end{pmatrix}.
	\end{split}
	\end{equation} Using complex notation \begin{equation*}
	\begin{split}
	V(z) = F_0(z) + i H(F_0)(z) = \frac{z-i}{1+z^2} = \frac{1}{z+i},
	\end{split}
	\end{equation*} \begin{equation*}
	\begin{split}
	U(z) = f(z) + i d(z),
	\end{split}
	\end{equation*} and \begin{equation*}
	\begin{split}
	\hat{G}(z) = \hat{g}(z) + i \hat{h}(z),
	\end{split}
	\end{equation*} the above vector system can be simply re-written as \begin{equation*}
	\begin{split}
	-\frac{1}{z} \cdot \frac{1+iz}{1-iz} (U(z)-U(0)) + (U(z) - U(0))' = \frac{\hat{G}(z)}{z}.
	\end{split}
	\end{equation*} With this complex notation, it is easy to directly integrate \eqref{eq:Linsys_rewrite} to obtain the formula \begin{equation}\label{eq:complex_form}
	\begin{split}
	U(z) = U(0) + \frac{z(-2z + i(1-z^2))}{(1+z^2)^2}\cdot \left( c +
	\int_0^z (-2 - i \frac{1-w^2}{w} ) \frac{\hat{G}(w)}{w} dw \right),  
	\end{split}
	\end{equation} with some (complex) constant of integration $c $, and converting back to the real notation, this is
	 \begin{equation*}
	\begin{split}
	\begin{pmatrix}
	f \\ \hat{d} 
	\end{pmatrix} = \frac{z}{(1+z^2)^2} \begin{pmatrix}
	1-z^2 & -2z \\
	2z & 1-z^2
	\end{pmatrix}  \left( \begin{pmatrix}
	c_f \\ c_d
	\end{pmatrix} + \int_0^z \begin{pmatrix}
	(1-w^2)/w & 2 \\
	-2 & (1-w^2)/w
	\end{pmatrix} \begin{pmatrix}
	\hat{g}(w) \\ \hat{h}(w)
	\end{pmatrix} dw \right)
	\end{split}
	\end{equation*} with some constants $c_f$ and $c_d$. In the expression inside the large brackets, the integral term is of order $O(z)$ for $|z|$ small, so that $c_f$ and $c_d$ are precisely the derivatives of $f$ and $\hat{d}$ evaluated at $z = 0$. The latter is zero since $\hat{d}$ is even and, and $c_f$ is zero if we assume further that $f'(0) = 0$. In particular, this shows that the solution $f$ is unique. 
	
%	It follows that the equation $Lf = g$ has a one-dimensional kernel spanned by the function \begin{equation*}
%	\begin{split}
%	\frac{z(1-z^2)}{(1+z^2)^2},
%	\end{split}
%	\end{equation*} whose Hilbert transform equals \begin{equation*}
%	\begin{split}
%	\frac{2z^2}{(1+z^2)^2}. 
%	\end{split}
%	\end{equation*}
	
	To actually conclude that $f$ given by the above formula provides a solution of $Lf = g$,  it needs to be argued that $ Hf = d$. This follows simply by observing that the functions $V, U$, and $\hat{G}$ can be extended as holomorphic functions on the upper half-plane, and that the formula \eqref{eq:complex_form} actually defines a holomorphic function of $z$ on the upper half-plane, which has an odd real part and even imaginary part once restricted onto the real axis. 
\end{proof}

%\begin{remark}
%	When $f$ and $g$ are not necessarily odd, then the system has a solution if and only if $g'(0) + 2Hg(0) = 0$ and $-2g(0) + (Hg)'(0) = 0$ holds, and the kernel is two dimensional. 
%\end{remark}

\subsubsection{Estimates for the inverse}

	We need an estimate for the inverse $L^{-1}$ in $H^3(\mathbb{R})$. Let us begin by breaking $L^{-1}$ into pieces: for $\sigma \in \{-1,0,1\}$, define the operators for \textit{odd} functions $g \in H^2(\mathbb{R})$ via \begin{equation*}
	\begin{split}
	T^{1,\sigma}(g) &:= \frac{z(1-z^2)}{(1+z^2)^2} \int_0^z w^{\sigma} \left( \frac{g(w)}{w} - \frac{1}{1+w^2}g'(0) \right) dw  \\
	T^{2,\sigma}(g) &:= \frac{2z}{(1+z^2)^2} \int_0^z w^{\sigma} \left( \frac{g(w)}{w} - \frac{1}{1+w^2}g'(0) \right) dw 
	\end{split}
	\end{equation*} and \begin{equation*}
	\begin{split}
	S^{1,\sigma}(g) &:= \frac{z(1-z^2)}{(1+z^2)^2} \int_0^z w^{\sigma} \left( \frac{Hg(w)-Hg(0)}{w} + \frac{2w}{1+w^2}(Hg)(0) \right) dw  \\
	S^{2,\sigma}(g) &:= \frac{2z}{(1+z^2)^2} \int_0^z w^{\sigma} \left( \frac{Hg(w)-Hg(0)}{w} + \frac{2w}{1+w^2}(Hg)(0) \right) dw .
	\end{split}
	\end{equation*} Note that $L^{-1}$ in \eqref{eq:inverse_rep} is a linear combination of twelve operators $\{ T^{l,\sigma}, S^{l',\sigma'}\}$ with $l,l' \in \{ 1, 2\}$ and $\sigma,\sigma' \in \{-1,0,1\}$. 
	
	Recall that $\lambda_n$ and $F_n$ were chosen to satisfy \begin{equation*}
	\begin{split}
	F_n = L^{-1} \left( G_n(z) - \lambda_n \frac{z(1-z^2)}{(1+z^2)^2} \right),
	\end{split}
	\end{equation*} and since $G_n$ involves derivatives of $F_j$ for $j < n$, $G_n$ belongs to $H^2(\mathbb{R})$ and no better. It turns out that unfortunately the operators $T^{1,-1}$ and $S^{1,-1}$ are not bounded from $H^2(\mathbb{R})$ to $H^3(\mathbb{R})$ (although naively one would expect that these operators gain one derivative), and to actually deduce that $F_n \in H^3(\mathbb{R})$, we shall need to use the specific form of $G_n$.
	
	Let us begin with the $L^2$ estimates. 
	
	\begin{lemma}[$L^2$-bounds]\label{lem:L2}
		Assume that $f$ and $g$ are odd functions on $\mathbb{R}$. Then we have the following estimates: \begin{equation}\label{eq:L^2_I}
		\begin{split}
		\V T^{l,\sigma} g \V_{L^2(\mathbb{R})} + \V S^{l,\sigma} g \V_{L^2(\mathbb{R})} \le C \V g\V_{H^2(\mathbb{R})},
		\end{split}
		\end{equation}  \begin{equation}\label{eq:L^2_II}
		\begin{split}
		\V T^{l,\sigma} (z g') \V_{L^2(\mathbb{R})} + \V S^{l,\sigma} (zg') \V_{L^2(\mathbb{R})} \le C \V g\V_{H^2(\mathbb{R})},
		\end{split}
		\end{equation} and \begin{equation}\label{eq:L^2_III}
		\begin{split}
		\V T^{l,\sigma} (\Lambda^{-1}(f) \cdot g') \V_{L^2(\mathbb{R})} + \V S^{l,\sigma} (\Lambda^{-1}(f) \cdot  g') \V_{L^2(\mathbb{R})} \le C\V f\V_{H^1(\mathbb{R})} \cdot \V g\V_{H^2(\mathbb{R})}
		\end{split}
		\end{equation}
	\end{lemma}

	\begin{proof}
		We shall first establish all the inequalities in the case $l = 1$. During the course of the argument, it will become clear that the case $l = 2$ can be treated in the same way and is only simpler. 
		
		Let us first show \eqref{eq:L^2_I}. We consider the most difficult case of $\sigma = -1$ first. We have \begin{equation*}
		\begin{split}
		T^{1,-1}(g) = \frac{z(1-z^2)}{(1+z^2)^2} \int_0^z \frac{g(w) - wg'(0)}{w^2} + \frac{w}{1+w^2} g'(0) dw 
		\end{split}
		\end{equation*} and then the second term is simply estimated by \begin{equation*}
		\begin{split}
		\V \frac{z(1-z^2)}{(1+z^2)^2} \cdot g'(0) \cdot \int_0^z \frac{w}{1+w^2} dw \V_{L^2} \le C\V g\V_{H^2} \cdot \V \frac{z(1-z^2) \ln(1+z^2) }{(1+z^2)^2} \V_{L^2} \le C\V g\V_{H^2}.
		\end{split}
		\end{equation*} For the other term, we obtain  \begin{equation*}
		\begin{split}
		&\V \frac{z(1-z^2)}{(1+z^2)^2} \int_0^z \frac{g(w)- wg'(0)}{w^2}dw \V_{L^2}   \le
		\V \frac{z^2(1-z^2)}{(1+z^2)^2} \V_{L^\infty}\cdot \V \frac{1}{z}\int_0^z \frac{g(w)- wg'(0)}{w^2}dw \V_{L^2} \\
		&\qquad\qquad\le C \V \frac{g(z)- zg'(0)}{z^2} \V_{L^2} = C \V -\partial_z \left( \frac{\int_0^z g'(w)dw}{z} \right) + \frac{1}{z} \int_0^z g''(w)dw \V_{L^2}  \le C\V g \V_{H^2},
		\end{split}
		\end{equation*} via applications of the Hardy inequality \eqref{eq:Hardy}. The argument for the case of $S^{1,-1}(g)$ is strictly analogous; we have \begin{equation*}
		\begin{split}
		S^{1,-1}(g) = \frac{z(1-z^2)}{(1+z^2)^2} \int_0^z \frac{Hg(w) - Hg(0)}{w^2} + \frac{2}{1+w^2}Hg(0) dw,
		\end{split}
		\end{equation*} and as before, the latter term can be estimated directly in $L^2$. For the first term, we may rewrite the integral as \begin{equation*}
		\begin{split}
		\int_0^z \frac{Hg(w) - Hg(0)}{w^2} dw = \int_0^z \left( -\partial_w \left(  \frac{1}{w}\int_0^w Hg'(x)dx  \right) + \frac{1}{w} \int_0^w Hg''(x) dx \right) dw 
		\end{split}
		\end{equation*} and then proceed exactly as before. 
		
		We consider the cases $\sigma = 0, 1$. When $\sigma = 0$, \begin{equation*}
		\begin{split}
		T^{1,0}(g) = \frac{z(1-z^2)}{(1+z^2)^2} \int_0^z \frac{g(w)}{w} - \frac{1}{1+w^2} g'(0) dw
		\end{split}
		\end{equation*} and note that the term involving $g'(0)$ is certainly bounded in $L^2$. For the other term, we simply use Hardy inequalities to obtain \begin{equation*}
		\begin{split}
		\V \frac{z(1-z^2)}{(1+z^2)^2} \int_0^z \frac{g(w)}{w}dw\V_{L^2} &\le
		\V \frac{z^2(1-z^2)}{(1+z^2)^2}\V_{L^\infty} \cdot \V\frac{1}{z} \int_0^z \left( \frac{1}{w}\int_0^w g'(x)dx \right) dw \V_{L^2} \\
		&\le C \V \frac{1}{z} \int_0^z g'(w) dw \V_{L^2} \le C \V g'\V_{L^2}. 
		\end{split}
		\end{equation*} In the case $\sigma = 1$, \begin{equation*}
		\begin{split}
		T^{1,1} (g) = \frac{z(1-z^2)}{(1+z^2)^2} \int_0^z g(w) - \frac{w}{1+w^2} g'(0) dw,
		\end{split}
		\end{equation*} and again, the term involving $g'(0)$ can be separately estimated in $L^2$. For the other term, \begin{equation*}
		\begin{split}
		\V \frac{z(1-z^2)}{(1+z^2)^2} \int_0^z g(w) dw\V_{L^2}  &\le
		\V \frac{z^2(1-z^2)}{(1+z^2)^2}\V_{L^\infty} \cdot \V \frac{1}{z} \int_0^z g(w)dw\V_{L^2}  \le C \V g\V_{L^2}.
		\end{split}
		\end{equation*} For each case of $\sigma = 0, 1$, the arguments for the $S^{1,\sigma}$ is again analogous and result in the same estimates. 
	
		\medskip
		
		Next, we deal with \eqref{eq:L^2_II}. First, for $T^{1,\sigma}$, we have \begin{equation*}
		\begin{split}
		T^{1,\sigma}(zg') = \frac{z(1-z^2)}{(1+z^2)^2} \int_0^z w^\sigma \left( g'(w) - g'(0) + \frac{w^2}{1+w^2} g'(0) \right) dw .
		\end{split}
		\end{equation*} 
		
		In the case $\sigma = - 1$, one can directly evaluate the integral in the second term, which results in the bound \begin{equation*}
		\begin{split}
		\V \frac{z(1-z^2)}{2(1+z^2)^2} \ln(1+z^2) g'(0) \V_{L^2} \le C \V g\V_{H^2}. 
		\end{split}
		\end{equation*} Regarding the first term, we rewrite it to bound  \begin{equation*}
		\begin{split}
		&\V \frac{z^2(1-z^2)}{(1+z^2)^2} \cdot \frac{1}{z} \int_0^z  \left( \frac{1}{w} \int_0^w g''(x) dx \right) dw \V_{L^2}\\ &\qquad\le \V \frac{z^2(1-z^2)}{(1+z^2)^2} \V_{L^\infty} \V \frac{1}{z} \int_0^z  \left( \frac{1}{w} \int_0^w g''(x) dx \right) dw \V_{L^2}\le C\V g\V_{H^2}
		\end{split}
		\end{equation*} via a repeated application of the Hardy inequality. 
		
		We treat the cases $\sigma = 0 , 1$ together; rewrite $T^{1,\sigma}(zg')$ as \begin{equation*}
		\begin{split}
		\frac{z(1-z^2)}{(1+z^2)^2} \int_0^z w^\sigma \left( g'(w) - \frac{1}{1+w^2} g'(0) \right) dw
		\end{split}
		\end{equation*} and note that the latter term can be directly estimated in $L^2$ by $\V g \V_{H^2}$ in both cases. When $\sigma = 0$, after integrating by parts we are left with simply \begin{equation*}
		\begin{split}
		\V \frac{z(1-z^2)}{(1+z^2)^2} g(z) \V_{L^2} \le C \V g\V_{L^2},
		\end{split}
		\end{equation*} and when $\sigma = 1$, we have \begin{equation*}
		\begin{split}
		&\V \frac{z(1-z^2)}{(1+z^2)^2} \left( zg(z) - \int_0^z g(w) dw \right) \V_{L^2} \\
		&\qquad \le \V \frac{z^2(1-z^2)}{(1+z^2)^2} \cdot g(z)\V_{L^2} + \V \frac{z^2(1-z^2)}{(1+z^2)^2} \cdot \frac{1}{z} \int_0^z g(w) dw \V_{L^2} \le C \V g\V_{L^2}. 
		\end{split}
		\end{equation*}
		
		We now turn to $S^{1,\sigma}(zg')$. We need to estimate in $L^2$ the following: \begin{equation*}
		\begin{split}
		S^{1,\sigma}(zg') = \frac{z(1-z^2)}{(1+z^2)^2} \int_0^z w^\sigma Hg'(w) dw
		\end{split}
		\end{equation*} since the Hilbert transform of $zg'$ is $ z H(g)'$, which vanishes at the origin. In the case $\sigma = -1$, we have (using the fact that $Hg'(0) = 0$) \begin{equation*}
		\begin{split}
			S^{1,-1}(zg') = \frac{z^2(1-z^2)}{(1+z^2)^2} \cdot \frac{1}{z}\int_0^z \left(
			\frac{1}{w}\int_0^w Hg''(x)dx  \right) dw
		\end{split}
		\end{equation*} which is bounded in $L^2$ by $C\V g \V_{H^2}$ by the Hardy inequality \eqref{eq:Hardy}. When $\sigma = 0$, \begin{equation*}
		\begin{split}
		S^{1,0}(zg') = \frac{z^2(1-z^2)}{(1+z^2)^2} \cdot \frac{1}{z}\int_0^z (Hg)'(w) dw,
		\end{split}
		\end{equation*} and this time, $\V S^{1,0}(zg')\V_{L^2}$ is bounded by $\V g\V_{H^1}$. Finally, in the remaining case of $\sigma = 1$, one may integrate by parts to obtain \begin{equation*}
		\begin{split}
		S^{1,1}(zg') = \frac{z(1-z^2)}{(1+z^2)^2} \left( -\int_0^z Hg(w) dw + z Hg(z) \right),
		\end{split}
		\end{equation*} and note that each term is bounded by a constant multiple of $\V g \V_{L^2}$ in $L^2$. 
		
		\medskip
		
		Turning to \eqref{eq:L^2_III}, we first consider the case of $T^{1,-1}(\Lambda^{-1}f \cdot g')$. The derivative of  $\Lambda^{-1}f \cdot g'$ at the origin is $Hf(0) \cdot g'(0)$. We need to bound \begin{equation*}
		\begin{split}
		\frac{z(1-z^2)}{(1+z^2)^2} \left( \int_0^z \frac{\int_0^w Hf(x)dx \cdot g'(w) - w Hf(0)g'(0)}{w^2} dw + Hf(0)g'(0) \cdot \int_0^z \frac{w}{1+w^2} dw \right)
		\end{split}
		\end{equation*} in $L^2$, and it is straightforward to estimate the second term. Regarding the first term, it suffices by an application of the Hardy inequality to estimate \begin{equation*}
		\begin{split}
		\frac{1}{w^2} \left( \int_0^w Hf(x)dx \cdot g'(w) - w Hf(0)g'(0) \right) 
		\end{split}
		\end{equation*} in $L^2$, and we first rewrite it as \begin{equation*}
		\begin{split}
		\frac{1}{w^2} \left( \int_0^w Hf(x)dx - wHf(0) \right) \cdot g'(w)  + \frac{1}{w} Hf(0) \cdot \left(  g'(w) - g'(0)  \right).
		\end{split}
		\end{equation*} Then, it is clear that the latter is bounded in $L^2$ by $C \V f\V_{H^1} \cdot \V g\V_{H^2}$. Next, the former can be further re-written as $g'(w)$ multiplied with \begin{equation*}
		\begin{split}
		-\partial_w \left( \frac{1}{w} \int_0^w Hf(x) dx \right) + \frac{1}{w}\left( Hf(w) - Hf(0) \right) = -\partial_w \left( \frac{1}{w} \int_0^w Hf(x) dx \right) + \frac{1}{w}\int_0^w Hf'(x)dx,
		\end{split}
		\end{equation*} each of which is bounded by $C \V f\V_{H^1}$. We omit the argument for the (simpler) cases of $T^{1,\sigma}(\Lambda^{-1}f \cdot g')$ with $\sigma = 0, 1$. 
		
		Next, regarding $S^{1,-1}(\Lambda^{-1}f\cdot g')$, we need an $L^2$ bound on \begin{equation*}
		\begin{split}
		& \frac{z(1-z^2)}{(1+z^2)^2} \int_0^z  \frac{H(\Lambda^{-1}f\cdot g')(w) - H(\Lambda^{-1}f \cdot g')(0)}{w^2} + \frac{2w}{1+w^2} H(\Lambda^{-1}f \cdot g')(0) dw.
		\end{split}
		\end{equation*} As before, the latter term can be estimated directly by evaluating the integral; \begin{equation*}
		\begin{split}
		\V \frac{z(1-z^2)}{(1+z^2)^2} \ln(1 + z^2)  H(\Lambda^{-1}f \cdot g')(0) \V_{L^2} \le C | H(\Lambda^{-1}f \cdot g')(0) | \le C \V f\V_{L^2} \V g\V_{H^1}, 
		\end{split}
		\end{equation*} where we have used \begin{equation}\label{eq:Hilbert_L^infty}
		\begin{split}
		\left| H(\Lambda^{-1}(f)\cdot g') (0) \right| &=\frac{2}{\pi} \left| \int_0^\infty \frac{1}{z}\Lambda^{-1}(f)(z) \cdot g'(z) dz \right|  \le C \V f \V_{L^2} \cdot \V g \V_{H^1}.
		\end{split}
		\end{equation} For the former, we first use the identity \eqref{eq:Hilbert_divid_z} to rewrite it as \begin{equation*}
		\begin{split}
		&\frac{z(1-z^2)}{(1+z^2)^2} \int_0^z  \frac{H(\Lambda^{-1}f\cdot g')(w) - H(\Lambda^{-1}f \cdot g')(0)}{w^2} dw \\ &\qquad = \frac{z^2(1-z^2)}{(1+z^2)^2} \cdot \frac{1}{z} \int_0^z \frac{1}{w} H\left( \frac{\Lambda^{-1}f(x)}{x} \cdot g'(x) \right)(w) dw.
		\end{split}
		\end{equation*} Then it suffices to estimate in $L^2$ the following function: \begin{equation*}
		\begin{split}
		\frac{1}{w} H\left( \frac{\Lambda^{-1}f(x)}{x} \cdot g'(x) \right)(w) = \frac{1}{w} \int_0^w H\left(\frac{\Lambda^{-1}f(x)}{x} \cdot g'(x)\right)'(s)ds,
		\end{split}
		\end{equation*} where we have used that the function \begin{equation*}
		\begin{split}
		\frac{\Lambda^{-1}f(x)}{x} \cdot g'(x)
		\end{split}
		\end{equation*} is even, so that its Hilbert transform is odd. At this point, we note that \begin{equation*}
		\begin{split}
		\V \frac{1}{w} \int_0^w H\left(\frac{\Lambda^{-1}f(x)}{x} \cdot g'(x)\right)'(s)ds \V_{L^2} \le C \V \frac{\Lambda^{-1}f(x)}{x} \cdot g'(x) \V_{H^1} \le C \V f\V_{H^1} \cdot \V g\V_{H^2}. 
		\end{split}
		\end{equation*} As in the previous cases, the arguments for $S^{1,\sigma}$ with $\sigma = 0, 1$ are simpler and we omit the proof. 
		
		\medskip
		
		Lastly, it only remains to treat the case when $l = 2$. However, note that the prefactor in this case equals \begin{equation*}
		\begin{split}
		\frac{2z^2}{(1+z^2)^2},
		\end{split}
		\end{equation*} which decays faster as $|z| \rightarrow \infty$ and also has one more factor of $z$ when $|z| \rightarrow 0$. Therefore, the arguments from the previous case simply carries over. 		
	\end{proof}

	We now state and prove the necessary $H^3$-bounds.

	\begin{lemma}[$H^3$-bounds]\label{lem:H3}
		Let $E$ and $F$ be odd functions in $H^3(\mathbb{R})$. For each $l \in \{1,2\}$ and $\sigma \in \{-1,0,1\}$, we have the estimates 
		\begin{equation}\label{eq:key_est1}
		\begin{split}
		\V S^{l,\sigma}((\Lambda^{-1}E) \cdot F') \V_{H^3(\mathbb{R})} + \V T^{l,\sigma}((\Lambda^{-1}E) \cdot F') \V_{H^3(\mathbb{R})} &\le C \V E \V_{H^3(\mathbb{R})} \V F\V_{H^3(\mathbb{R})},
		\end{split}
		\end{equation}\begin{equation}\label{eq:key_est2}
		\begin{split}
		\V S^{l,\sigma}(z F') \V_{H^3(\mathbb{R})} + \V T^{l,\sigma}(z F') \V_{H^3(\mathbb{R})} &\le C \V F\V_{H^3(\mathbb{R})},
		\end{split}
		\end{equation} and \begin{equation}\label{eq:key_est3}
		\begin{split}
		\V S^{l,\sigma}(E\cdot H(F)) \V_{H^3(\mathbb{R})} + \V T^{l,\sigma}( {E} \cdot H(F)) \V_{H^3(\mathbb{R})} &\le C \V E \V_{H^3(\mathbb{R})} \V F\V_{H^3(\mathbb{R})}. 
		\end{split}
		\end{equation}
	\end{lemma}
	
	\begin{proof}
		Let us give a brief outline of the proof. We first establish the inequality \eqref{eq:key_est1}, and then we proceed to the proof of \eqref{eq:key_est2}, which can be done in a similar way. After that we prove \eqref{eq:key_est3}. We fix $l = 1$, since as in the previous Lemma, the $l = 2$ case is only simpler. 
		 % Moreover, we note in advance that for each of the cases \eqref{eq:key_est1}, \eqref{eq:key_est2}, and \eqref{eq:key_est3}, the case of $\sigma = -1$ is the most difficult one (for $\sigma = 0, 1$, we can actually put $H^2$ of $E$ and $F$ in the right hand side of the estimates).
		
		\medskip
		
		\textbf{(i) proof of \eqref{eq:key_est1}}
		
		\smallskip
		
		Take $\sigma = -1$, and let us establish \eqref{eq:key_est1} starting with the term involving $T^{1,-1}$. Noting that $\partial_z(\Lambda^{-1}E \cdot F')(0) = (\Lambda^{-1}E)'(0) F'(0) = H(E)(0) F(0)$, we have to bound \begin{equation*}
		\begin{split}
		&\V \partial_z^3 \left(  \frac{z(1-z^2)}{(1+z^2)^2} \int_0^z \frac{\Lambda^{-1}E(w)\cdot F'(w) - w H(E)(0) F'(0) }{w^2} + \frac{w}{1+w^2} H(E)(0) F'(0) dw    \right) \V_{L^2}. 
		\end{split}
		\end{equation*}
		
		To begin with, the second term is clearly bounded in $H^3$ by $C\V E\V_{H^3} \cdot \V F\V_{H^3}$. Regarding the first term, from the previous $L^2$-bound and an interpolation, it suffices to treat the case where all three derivatives fall on the integral term. Taking one derivative, we get \begin{equation*}
		\begin{split}
		&\frac{1-z^2}{(1+z^2)^2} \cdot \frac{\Lambda^{-1}E(z)\cdot F'(z) - z H(E)(0) F'(0) }{z} \\
		&\qquad\qquad= \frac{1-z^2}{(1+z^2)^2} \left(  \frac{\Lambda^{-1}E(z) }{z} F'(z) - H(E)(0)F'(0)   \right)
		\end{split}
		\end{equation*}  and then using the algebra property $\V fg\V_{H^2} \le C\V f\V_{H^2} \V g\V_{H^2}$ gives \begin{equation*}
		\begin{split}
		\V \frac{1-z^2}{(1+z^2)^2} \cdot \frac{1}{z} \int_0^z HE(w) dw  \cdot F'(z)\V_{H^2} \le C \V \frac{1}{z} \int_0^z HE(w) dw \V_{H^2} \cdot \V F'\V_{H^2} \le C\V E\V_{H^2}\V F\V_{H^3}
		\end{split}
		\end{equation*} with the Hardy inequality \eqref{eq:Hardy}. 
		
		Next, for $S^{1,-1}$, we need a bound on \begin{equation*}
		\begin{split}
		&\V \frac{z(1-z^2)}{(1+z^2)^2} \int_0^z  \frac{H(\Lambda^{-1}E\cdot F')(w) - H(\Lambda^{-1}E \cdot F')(0)}{w^2} + \frac{2w}{1+w^2} H(\Lambda^{-1}E \cdot F')(0) dw \V_{H^3}. 
		\end{split}
		\end{equation*} The latter term can be estimated directly; integrating in $z$ gives \begin{equation*}
		\begin{split}
		\V \frac{z(1-z^2)}{(1+z^2)^2} \ln(1 + z^2)  H(\Lambda^{-1}E \cdot F')(0) \V_{H^3} \le C | H(\Lambda^{-1}E \cdot F')(0) | \le C \V E\V_{H^3} \V F\V_{H^3}, 
		\end{split}
		\end{equation*} where we have used \eqref{eq:Hilbert_L^infty}. For the other term, we first use the identity \eqref{eq:Hilbert_divid_z} to first rewrite it as \begin{equation*}
		\begin{split}
		&\frac{z(1-z^2)}{(1+z^2)^2} \int_0^z  \frac{H(\Lambda^{-1}E\cdot F')(w) - H(\Lambda^{-1}E \cdot F')(0)}{w^2} dw \\ &\qquad = \frac{z(1-z^2)}{(1+z^2)^2} \int_0^z \frac{1}{w} H\left( \frac{\Lambda^{-1}E(x)}{x} \cdot F'(x) \right)(w) dw.
		\end{split}
		\end{equation*} As before, we only need to consider the case when the derivatives fall on the integral term, and taking one derivative, we obtain \begin{equation*}
		\begin{split}
		\frac{1-z^2}{(1+z^2)^2} H\left( \frac{\Lambda^{-1}E(w)}{w}\cdot F'(w) \right) (z) ,
		\end{split}
		\end{equation*} and since \begin{equation*}
		\begin{split}
		\V\frac{1-z^2}{(1+z^2)^2} H\left( \frac{\Lambda^{-1}E(w)}{w}\cdot F'(w) \right) (z)\V_{H^2} &\le C\V \frac{\Lambda^{-1}E(z)}{z}\cdot F'(z)\V_{H^2} ,
		\end{split}
		\end{equation*} we conclude the desired bound as in the above case of $T^{1,-1}$. This concludes the estimate \eqref{eq:key_est1} in the case $\sigma = -1$. 
		
		We now deal with the case $\sigma = 0$. We need to bound \begin{equation*}
		\begin{split}
		&\frac{z(1-z^2)}{(1+z^2)^2} \int_0^z \frac{\Lambda^{-1}E(w)}{w} \cdot F'(w) - \frac{1}{1+w^2} H(E)(0) F'(0) dw
		\end{split}
		\end{equation*} in $H^3$, and again we only worry about the case when derivatives fall on the integral. Taking one derivative, \begin{equation*}
		\begin{split}
		\frac{(1-z^2)}{(1+z^2)^2}  \left( \Lambda^{-1}E(z) \cdot F'(z)   - \frac{z}{1+z^2} H(E)(0)F'(0) \right)
		\end{split}
		\end{equation*} and it is direct to see that \begin{equation*}
		\begin{split}
		\V \frac{(1-z^2)}{(1+z^2)^2} \cdot \frac{z}{1+z^2} H(E)(0)F'(0)\V_{H^2}  \le C \V E\V_{H^3} \cdot \V F\V_{H^3}.
		\end{split}
		\end{equation*} Regarding the other term, we rewrite it as \begin{equation*}
		\begin{split}
			\frac{z(1-z^2)}{(1+z^2)^2} \cdot F'(z) \cdot \frac{1}{z} \int_0^z H(E)(w) dw ,
		\end{split}
		\end{equation*} and then we have \begin{equation*}
		\begin{split}
		\Vert 	\frac{z(1-z^2)}{(1+z^2)^2} \cdot F'(z) \cdot \frac{1}{z} \int_0^z H(E)(w) dw \V_{H^2} \le C\V F\V_{H^3} \cdot \V E\V_{H^2},
		\end{split}
		\end{equation*} using the Hardy inequality. Now the remaining case $\sigma = 1$ can be handled similarly; this time we have \begin{equation*}
		\begin{split}
		\frac{z(1-z^2)}{(1+z^2)^2} \int_0^z \Lambda^{-1}E(w) \cdot F'(w) - \frac{w}{1+w^2} H(E)(0)F'(0) dw,
		\end{split}
		\end{equation*} and note that the latter term is bounded in $H^3$. The other term, when a derivative falls on the integral, becomes \begin{equation*}
		\begin{split}
		\frac{z^2(1-z^2)}{(1+z^2)^2} \frac{\Lambda^{-1}E(z)}{z} \cdot F'(z)  = \frac{z^2(1-z^2)}{(1+z^2)^2} \frac{1}{z} \int_0^z H(E)(w)dw  \cdot F'(z) 
		\end{split}
		\end{equation*} and again using the Hardy inequality, \begin{equation*}
		\begin{split}
		\V \frac{z^2(1-z^2)}{(1+z^2)^2} \frac{1}{z} \int_0^z H(E)(w)dw  \cdot F'(z) \V_{H^2} \le C \V E\V_{H^2} \cdot \V F\V_{H^3}.
		\end{split}
		\end{equation*}
		
		Lastly we deal with $S^{1,0}$ and $S^{1,1}$. In these cases, for simplicity set $g= \Lambda^{-1}E \cdot F' $ and we need to bound \begin{equation*}
		\begin{split}
		\frac{z(1-z^2)}{(1+z^2)^2} \int_0^z w^\sigma \left( \frac{Hg(w)}{w} + \left( \frac{2w}{1+w^2} - \frac{1}{w} \right) Hg(0)  \right)dw 
		\end{split}
		\end{equation*} for $\sigma = 0,1$ and differentiating the integral, \begin{equation*}
		\begin{split}
		\frac{1-z^2}{(1+z^2)^2} \left( z^\sigma  Hg(z) + z^\sigma \left( \frac{2z^2}{1+z^2} -1 \right)Hg(0) \right),
		\end{split}
		\end{equation*} and it is straightforward to see that when $\sigma = 0, 1$, both terms are bounded in $H^2$ by $C \V F\V_{H^3} \cdot \V E \V_{H^3}$. This establishes \eqref{eq:key_est1}.

		\medskip
		
		\textbf{(ii) proof of \eqref{eq:key_est2}}
		
		\smallskip

		In this case, $H(zF') = zH(F)'$ and hence the corresponding bound for $S^{l,\sigma}$ follows similarly from the bound for $T^{l,\sigma}$. We first take $\sigma = -1$, and proceed to show that \begin{equation*}
		\begin{split}
		\V \partial_z^3 \left( \frac{z(1-z^2)}{(1+z^2)^2} \int_0^z \frac{wF'(w) - wF'(0)}{w^2} + \frac{w}{1+w^2} F'(0) \right) \V_{L^2} \le C\V F\V_{H^3}
		\end{split}
		\end{equation*} holds. Taking one derivative on the integral term gives \begin{equation*}
		\begin{split}
		\frac{1-z^2}{(1+z^2)^2} \cdot \left( F'(z)  + \left( \frac{z^2}{1+z^2}-1 \right) F'(0) \right),
		\end{split}
		\end{equation*} and it is straightforward to see that both terms are bounded by $C\V F\V_{H^3}$. 
		
		On the other hand, when $\sigma = 1$, we need a bound on \begin{equation*}
		\begin{split}
		\V \partial_z^3 \left(  \frac{z(1-z^2)}{(1+z^2)^2} \left( \int_0^z wF'(w) - wF'(0) + \frac{w^3}{1+w^2} F'(0) dw  \right)  \right)\V_{H^2},
		\end{split}
		\end{equation*} and taking one derivative on the integral term gives \begin{equation*}
		\begin{split}
		\frac{z(1-z^2)}{(1+z^2)^2} \left( zF'(z) + \left(\frac{z^3}{1+z^2} - z \right) F'(0) \right) = \frac{z^2(1-z^2)}{(1+z^2)^2} \left( F'(z) - \frac{1}{1+z^2} F'(0) \right),
		\end{split}
		\end{equation*} so that \begin{equation*}
		\begin{split}
		\V \frac{z^2(1-z^2)}{(1+z^2)^2} \left( F'(z) - \frac{1}{1+z^2} F'(0) \right)\V_{H^2} \le C \V F'(z) - \frac{1}{1+z^2} F'(0) \V_{H^2} \le C \V F\V_{H^3}. 
		\end{split}
		\end{equation*}
		
		We omit the proof for the intermediate case of $\sigma = 0$, which is simpler. 
			
		\medskip
		
		\textbf{(iii) proof of \eqref{eq:key_est3}}
		
		\smallskip
		
		Noting that the function $G := H(E)\cdot F$ is again odd, it suffices to show that $\V T^{l,\sigma}(G) \V_{H^3} + \V S^{l,\sigma}(G) \V_{H^3} \le C\V G\V_{H^3}$. 
		
		Taking the derivative of the integral in \begin{equation*}
		\begin{split}
		T^{1,-1}(G) = \frac{z(1-z^2)}{(1+z^2)^2} \int_0^z \frac{1}{w} \left( \frac{G(w)}{w} - \frac{1}{1+w^2}G'(0) \right) dw,
		\end{split}
		\end{equation*} we obtain \begin{equation*}
		\begin{split}
		\frac{1-z^2}{(1+z^2)^2} \left( \frac{1}{z}\int_0^z G'(w)dw - \frac{1}{1+z^2} G'(0) \right),
		\end{split}
		\end{equation*} which is clearly bounded in $H^2$ by $C\V G\V_{H^3}$.
		
		Similarly, differentiating the integral term in \begin{equation*}
		\begin{split}
		S^{1,-1}(G) = \frac{z(1-z^2)}{(1+z^2)^2} \int_0^z \frac{1}{w} \left( \frac{H(G)(w) - H(G)(0)}{w} + \frac{2w}{1+w^2} H(G)(0) \right) dw,
		\end{split}
		\end{equation*} we get \begin{equation*}
		\begin{split}
		\frac{1-z^2}{(1+z^2)^2} \left( \frac{1}{z}\int_0^z H(G)'(w)dw + \frac{2z}{1+z^2} H(G)(0) \right),
		\end{split}
		\end{equation*} which is again bounded in $H^2$ by $C\V G\V_{H^3}$. 
		
		As before, the cases $\sigma = 0, 1$ are only simpler and we omit the proof. This concludes the proof of \eqref{eq:key_est3}. 
	\end{proof}

	Given the above lemmas, we are in a position to establish the convergence of \eqref{eq:expansion} for $|a|$ small, which completes the proof of Theorem \ref{thm:main1}.

	\begin{proof}[Proof of Theorem \ref{thm:main1}]
	It suffices to show that there exists an absolute constant $r>0$, such that for all $n \ge 1$, we have \begin{equation*}
	\begin{split}
	\V F_n \V_{H^3(\mathbb{R})} \le r^n, \qquad |\lambda_n| \le r^n. 
	\end{split}
	\end{equation*} Then we may pick $a_0 = 1/r$, and take $a_0$ smaller if necessary to guarantee that $\lambda(a) > -2$ for $|a| < a_0$. For simplicity we set $\V F_n \V_{H^3} =: \mu_n$ and let us translate the estimates from Lemma \ref{lem:L2} and Lemma \ref{lem:H3} in terms of sequences $\{ \mu_n \}, \{ \lambda_n \}$.

	Let us begin by recalling that \begin{equation*}
	\begin{split}
	L(F_n) = G_n -  \lambda_n \frac{z(1-z^2)}{(1+z^2)^2} ,
	\end{split}
	\end{equation*} with \begin{equation*}
	\begin{split}
	G_n(z) =  \sum_{j=0}^{n-1} \Lambda^{-1}(F_j)(z)  \cdot F'_{n-1-j}(z) - \sum_{j=1}^{n-1} z \lambda_j F'_{n-j}(z) - 2 \sum_{j=1}^{n-1} H(F_{n-j})(z) F_{j}(z),
	\end{split}
	\end{equation*} and \begin{equation*}
	\begin{split}
	\lambda_n = G_n'(0) + 2(HG_n)(0).
	\end{split}
	\end{equation*} Noting that \begin{equation*}
	\begin{split}
	G_n'(0) &= \sum_{j=0}^{n-1} HF_j(0) F'_{n-1-j}(0) - \sum_{j=1}^{n-1} \lambda_j F'_{n-j}(0) - 2\sum_{j=1}^{n-1} F_j'(0) HF_{n-j}(0) \\
	HG_n(0) &= \sum_{j=0}^{n-1} H(\Lambda^{-1}(F_j) F'_{n-1-j})(0) - \sum_{j=1}^{n-1} \left(  -F_j(0) F_{n-j}(0) + HF_{j}(0) HF_{n-j}(0)     \right)
	\end{split}
	\end{equation*} and simply using crude bounds $\V H(F_j)\V_{L^\infty}, \V F_j'\V_{L^\infty} \le C \V F_j\V_{H^3}$ as well as \eqref{eq:Hilbert_L^infty} we deduce that \begin{equation}\label{eq:A}
	\begin{split}
	|\lambda_n| \le C\mu_0 \mu_{n-1} +  C \sum_{j=1}^{n-1} \left( \mu_j \mu_{n-1-j} + |\lambda_j| \mu_{n-j} + \mu_j \mu_{n-j} \right).
	\end{split}
	\end{equation}
	
	Given $\lambda_n$, we first write $L^{-1}( G_n -  \lambda_n \frac{z(1-z^2)}{(1+z^2)^2})$ as a linear combination (with constant coefficients) of twelve operators $T^{l,\sigma}$, $S^{l',\sigma'}$ with $l,l'\in \{1,2\}, \sigma,\sigma' \in \{-1,0,1 \}$, and applying Lemmas \ref{lem:L2}, \ref{lem:H3} to each term of $G_n$, we  deduce that \begin{equation}\label{eq:B}
	\begin{split}
	\mu_n \le C|\lambda_n| +  C\mu_0 \mu_{n-1} +  C \sum_{j=1}^{n-1} \left( \mu_j \mu_{n-1-j} + |\lambda_j| \mu_{n-j} + \mu_j \mu_{n-j} \right).
	\end{split}
	\end{equation}
		
%	We recall that $\hat{g}$ and $\hat{h}$ are explicitly given by \begin{equation}\label{eq:tilde_g}
%	\begin{split}
%	&\hat{g}(z) = -\frac{\tan^{-1}(z)-z}{z} F'_{n-1} - \lambda_n \frac{z(1-z^2)}{(1+z^2)^2} +  \sum_{j=1}^{n-1} \left(   
%	\frac{u_{F_j}}{z} F'_{n-1-j} - \lambda_j F'_{n-j} - 2 F_j HF_{n-j} \right) \\
%	& -\frac{1}{1+z^2} \left( HF_0(0) F'_{n-1}(0) - \lambda_n - \sum_{j=1}^{n-1} \left(  HF_j(0) F'_{n-1-j}(0) - \lambda_j F'_{n-j}(0) - 2 F'_j(0) HF_{n-j}(0)  \right)        \right)
%	\end{split}
%	\end{equation} 

	Therefore, combining \eqref{eq:A} and \eqref{eq:B}, there exists some absolute constant $C_0 > 0$ such that $\zeta_n:= \mu_n + C_0(|\lambda_n| + \mu_{n-1}) $ for $n \ge 1$ satisfies the inequality \begin{equation}\label{eq:quad}
	\begin{split}
	\zeta_n \le C\sum_{j=1}^{n-1} \zeta_j \zeta_{n-j}
	\end{split}
	\end{equation} for all $n \ge 1$. Then, denoting $\bar{\zeta}_n$ to be the solution of the quadratic recursion \begin{equation}\label{eq:quad_recur}
	\begin{split}
	\bar{\zeta}_n :=  \sum_{j=1}^{n-1} \bar\zeta_j \bar\zeta_{n-j}, \qquad n \ge 2
	\end{split}
	\end{equation} with $\bar{\zeta}_1 := C\zeta_1$ (here $C$ is the absolute constant from \eqref{eq:quad}), we deduce the desired statement, simply because the sequence $\bar{\zeta}_n$ is precisely the sequence of coefficients of the Taylor expansion of $f$ satisfying \begin{equation*}
	\begin{split}
	f(z) - f(z)^2 = \bar{\zeta}_1 z
	\end{split}
	\end{equation*} around $z = 0$, so that in particular we have, for some $r > 0$, 
	 \begin{equation*}
	\begin{split}
	|\zeta_n| \le C^n|\bar{\zeta}_n| \le r^n
	\end{split}
	\end{equation*} where $C$ is the same constant from \eqref{eq:quad}. This finishes the proof.
\end{proof}

\begin{remark}
	As we have mentioned in the introduction, a similar argument goes through for $H^s$ with any value of $s \ge 3$, which in particular concludes that the functions $F_1, F_2,\cdots$, as well as the profile $F$ are $C^\infty(\mathbb{R})$-smooth. It is likely that these functions are indeed real analytic, and one way to establish such a result would be to carefully carry out the $H^s$-estimates, keeping track of the dependence of multiplicative constants in terms of $s$.  
\end{remark}

\begin{proof}[Proof of Theorem \ref{thm:main1prime}]
	We simply use the equation \eqref{eq:DeG_resc} to obtain the desired decay statement for the self-similar profile $F$.
	Multiplying both sides of \eqref{eq:DeG_resc} by $z^{\frac{1}{1+\lambda} - 1}$ and then integrating, we get: \begin{equation*}
	\begin{split}
	z^{\frac{1}{1+\lambda}} F(z) = \frac{1}{1+\lambda} \int_0^z s^{\frac{1}{1+\lambda} - 1 } \left( F(s) \cdot H(F)(s) - a \Lambda^{-1}(F)(s) \cdot F'(s) \right) ds.
	\end{split}
	\end{equation*} 	
	Since $a$ is small, let's assume $\lambda > -1$. Then the function $s^{\frac{1}{1+\lambda} - 1} $ is locally integrable.  We consider two cases: $1/(1+\lambda) -1 \ge 0$ and $< 0$. In the latter case, the right hand side is simply bounded in absolute value in the limit $z \rightarrow +\infty$, which gives the desired decay of $z^{-\frac{1}{1+\lambda}}$. To see this, just note that the integral in the region $0 \le s \le 1$ is bounded by a constant multiple of \begin{equation*}
	\begin{split}
	\V F \cdot H(F) - a \Lambda^{-1}(F) \cdot F' \V_{L^\infty},
	\end{split}
	\end{equation*} and for $s > 1$ we just use the $L^2$ bound on the first term\begin{equation*}
	\begin{split}
	\V F H(F)\V_{L^1}\leq \V F \V_{L^2} \cdot \V H(F)\V_{L^2}.
	\end{split}
	\end{equation*} 
	For the transport term we integrate by parts to get  $$az^{\frac{1}{1+\lambda}-1}\Lambda^{-1}(F)(z)\cdot F(z)-a\int_0^z s^{\frac{1}{1+\lambda}-1}H(F)F+(\frac{1}{1+\lambda}-1)s^{\frac{1}{1+\lambda}-1} \cdot \frac{1}{s}\Lambda^{-1}(F) \cdot F\, ds,$$ and since $|\Lambda^{-1}(F)(z)|\leq z^{\frac{1}{2}}$ due to the $L^2$ estimate on $H(F)$, the boundary term can be subsumed into the term on the left hand side for $z$ large, which is $z^{\frac{1}{1+\lambda}}F$. This takes care of the boundary term. Now, the two integral terms are handled just as before noting that $s^{-1}\Lambda^{-1}F\in L^2$ using Hardy's inequality. 
	
	In the other case of $1/(1+\lambda) \ge 1$, running the above argument gives instead the decay $|F(z)| \lesssim z^{-1}$, and in particular $F$ and $H(F)$ belongs to $L^p$ for all $p > 1$. Once one has this, we re-insert this information above to get better decay on $F$ until we see that $F$ actually decays like $z^{-\frac{1}{1+\lambda}}.$
\end{proof}

\section{Self-similar blow up for H\"older continuous data}\label{sec:Hoelder}

In this section, our goal consists of establishing Theorem \ref{thm:main2}, which states the self-similar blow up of merely $C^\alpha$ solutions to \eqref{eq:DeG}, over a range of $a$ scaling as $1/\alpha$ for $\alpha \rightarrow 0^+$. This is done first by constructing a family of self-similar solutions to the CLM equation which are merely $C^\alpha$; in fact, they are smooth functions of $|z|^\alpha$. Then the proof will be similar to the smooth case except that we will need to take great care in that we are dealing with smooth functions of $|z|^\alpha$ rather than simply smooth functions. This leads us to define the operators $\tilde H^{(n)}$ which describe how the Hilbert transform acts on $L^2$ functions of $|z|^\alpha$. We effectively prove that $L^2$ functions of $|z|^\alpha$ are mapped to $L^2$ functions of $|z|^\alpha$ with an $L^2$ operator norm on the order of $\frac{C}{\alpha},$ though we only do this for $\alpha=\frac{1}{n}$ with $n\in\mathbb{N}$, as all we actually need is a sequence of $C^\alpha$ solutions with $\alpha\rightarrow 0$. Other than the extra technical machinery needed to deal with this case as well as the careful checking of the dependence of constants on $\alpha$ as $\alpha\rightarrow 0$, the main idea is similar that in the preceding section. 

\subsection{Self-similar H\"older continuous solutions for CLM}

As in the case of smooth data, the starting point is to find self-similar solutions to the Constantin-Lax-Majda equation which is only H\"older continuous. More precisely, for each $ 0 < \alpha <1$ we seek for a $C^\alpha$-profile $F^{(\alpha)}_0$ such that \begin{equation*}
\begin{split}
\omega(t,x) = \frac{1}{1-t} F^{(\alpha)}_0\left( \frac{x}{(1-t)^{\frac{1}{\alpha}}} \right)
\end{split}
\end{equation*} provides a solution to the Constantin-Lax-Majda equation. This reduces to the following differential equation for $F^{(\alpha)}_0$:\begin{equation}\label{eq:alpha}
\begin{split}
F^{(\alpha)}_0 + \frac{1}{\alpha} zF^{(\alpha)'}_0 + 2F^{(\alpha)}_0 \cdot H(F^{(\alpha)}_0) = 0,
\end{split}
\end{equation} and it can be checked explicitly that the following pair of functions \begin{equation}\label{eq:F_alpha}
\begin{split}
F^{(\alpha)}_0(z) = \frac{\sin\left(\frac{\alpha\pi}{2}\right) \mathrm{sgn}(z) |z|^\alpha }{ 1 + 2\cos\left(\frac{\alpha\pi}{2}\right)|z|^\alpha + |z|^{2\alpha} }
\end{split}
\end{equation} and \begin{equation}\label{eq:HF_alpha}
\begin{split}
H(F^{(\alpha)}_0)(z) = - \frac{1 + \cos\left(\frac{\alpha\pi}{2}\right)|z|^\alpha}{ 1 + 2\cos\left(\frac{\alpha\pi}{2}\right)|z|^\alpha + |z|^{2\alpha} }.
\end{split}
\end{equation} provides a solution to \eqref{eq:alpha} (see Figure \ref{fig:C_alpha}). In the case $\alpha = 1$, we obtain our familiar pair of functions $(F_0, H(F_0))$. To argue that \eqref{eq:HF_alpha} is indeed the Hilbert transform of \eqref{eq:F_alpha}, we use complex notation \begin{equation*}
\begin{split}
V(z) := F^{(\alpha)}_0(z) + i H(F^{(\alpha)}_0)(z)
\end{split}
\end{equation*} and then taking the Hilbert transform of \eqref{eq:alpha}, one obtains the ODE system \begin{equation*}
\begin{split}
V + \frac{1}{\alpha} z V' - iV^2 = 0. 
\end{split}
\end{equation*} Then, integrating the system, one obtains that $V$ is necessarily in the following form with some (complex) constant of integration $C = C^{(\alpha)}$: \begin{equation*}
\begin{split}
V(z) = \frac{1}{i + Cz^\alpha},
\end{split}
\end{equation*} where we define $z^\alpha := r^\alpha e^{i\alpha\theta}$ for $ z= re^{i\theta}$ ($0 \le \theta \le \pi$) as a holomorphic function on the upper half-plane. It is easy to check that \begin{equation*}
\begin{split}
 C^{(\alpha)} = \sin\left(\frac{\alpha\pi}{2}\right) + i \cos\left(\frac{\alpha\pi}{2}\right)
\end{split}
\end{equation*} is the unique number (up to a scaling in $\mathbb{R}^+$) which makes $V(z)$ holomorphic in the upper half-plane with an \textit{odd} real part when restricted to the real axis. With this value of $C^{(\alpha)}$, one has \begin{equation*}
\begin{split}
F^{(\alpha)}_0(z) = \Re\left(  V \right)(z),\qquad H(F^{(\alpha)}_0)(z) = \Im\left( V \right)(z)
\end{split}
\end{equation*} which in particular concludes that \eqref{eq:HF_alpha} is the Hilbert transform of \eqref{eq:F_alpha}. We have proved the following proposition: \begin{proposition}
	For each $0 < \alpha \le 1$, $F^{(\alpha)}$ is the only odd data which is a smooth function of $z^\alpha$ on $\mathbb{R}^+$ and extends as holomorphically to the upper half-plane, which defines a self-similar solution to the Constantin-Lax-Majda equation \eqref{eq:CLM}. 
\end{proposition}

\begin{figure}
	\includegraphics[scale=0.3]{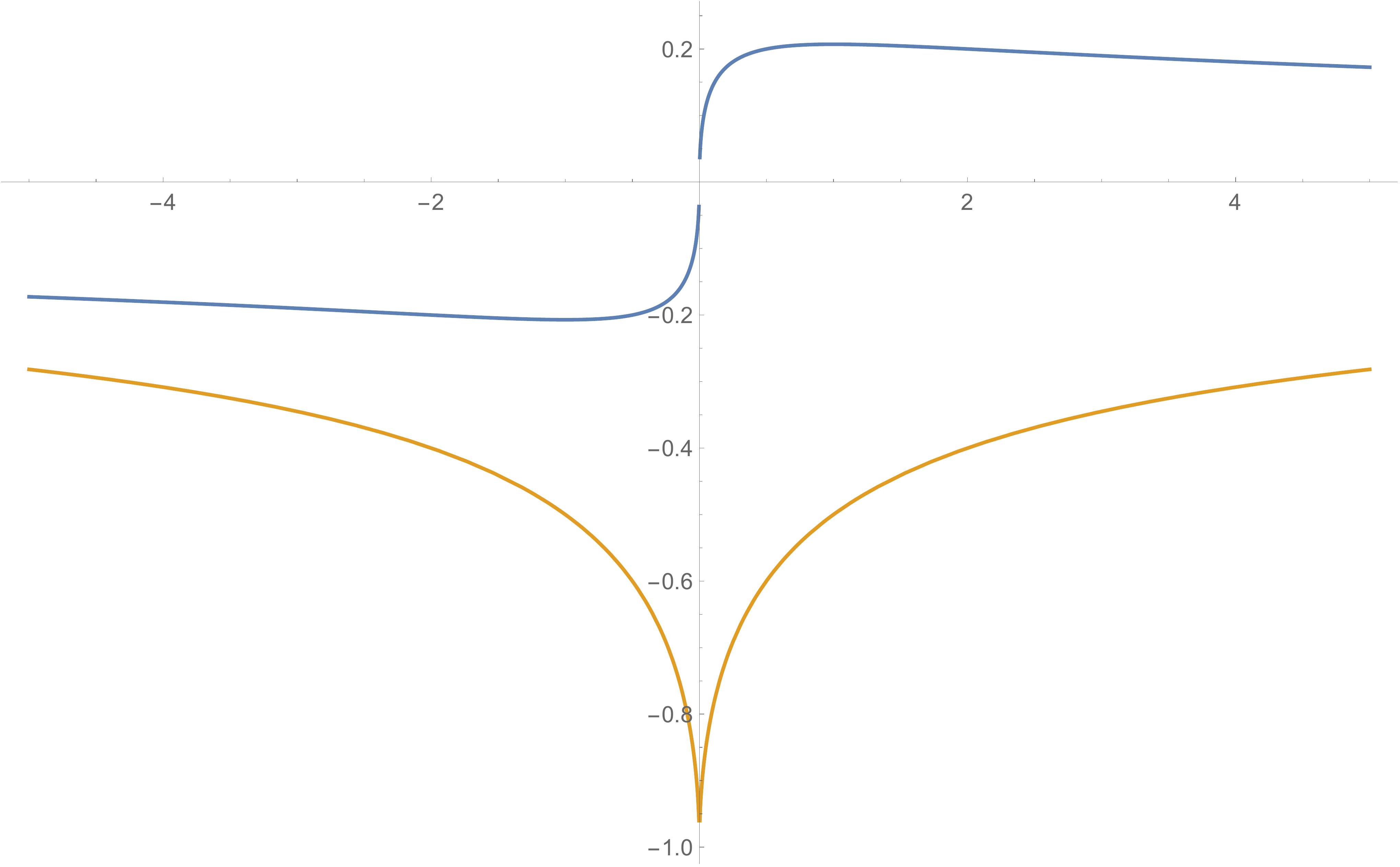} 
	\centering
	\caption{The functions $F^{(1/2)}$ (odd) and $H(F^{(1/2)})$ (even) plotted on $[-5,5]$.}
	\label{fig:C_alpha}
\end{figure}

\subsection{Self-similar H\"older continuous solutions for De Gregorio}

\subsubsection{The expansion in $a$ and the linear operator}

To seek $C^\alpha$ self-similar solutions to the De Gregorio equation, we now take the ansatz \begin{equation*}
\begin{split}
\omega(t,x) = \frac{1}{1-t} F \left( \frac{x}{(1-t)^{\frac{1+  \lambda(a)}{\alpha}}} \right), \qquad \lambda(0) = 0,
\end{split}
\end{equation*} as in the smooth case. In terms of $F$, we obtain  \begin{equation}\label{eq:DeG_resc_alpha}
\begin{split}
F(z) + \left( \frac{1+\lambda(a)}{\alpha} z + a \cdot  u_F(z) \right) F'(z) + 2F(z) HF(z) = 0. 
\end{split}
\end{equation} We expand $F$ and $\lambda$ again as \begin{equation}\label{eq:expansion_alpha}
\begin{split}
F(z) &= F^{(\alpha)}_0(z) + \sum_{n=1}^\infty a^n F_n(z),\qquad  \lambda(a) = \sum_{n=1}^\infty a^n \lambda_n
\end{split}
\end{equation} Each term in the expansions depend on $\alpha$ but we suppress from writing out the dependence.

Inserting the expansion \eqref{eq:expansion_alpha} into \eqref{eq:DeG_resc_alpha}, for each $0 < \alpha < 1$, we obtain with $F_0 = F^{(\alpha)}_0$ \begin{equation}\label{eq:identities_alpha}
\begin{split}
F_n + \frac{1}{\alpha}zF_n' + 2H(F_0) F_n + 2F_0 H(F_n) =  \sum_{j=0}^{n-1} \Lambda^{-1}(F_j) F'_{n-1-j} - \frac{1}{\alpha} \sum_{j=1}^n \lambda_j zF_{n-j}' - 2 \sum_{j=1}^{n-1} F_j H(F_{n-j})
\end{split}
\end{equation} for each $n \ge 1$. Just as in the case of smooth solutions, we write \eqref{eq:identities_alpha} in the form \begin{equation*}
\begin{split}
L(F_n) = G_n - \frac{1}{\alpha} \lambda_n zF_0' ,\qquad G_n = G_n(F_1,\cdots,F_{n-1};\lambda_1,\cdots,\lambda_{n-1})
\end{split}
\end{equation*} with the linear operator $L = L^{(\alpha)}$ defined by  \begin{equation*}
\begin{split}
Lf = f + \frac{1}{\alpha} zf' + 2H(F_0) f + 2F_0 H(f).
\end{split}
\end{equation*} For each $n \ge 1$, there is a unique $\lambda_n$ which makes the linear system solvable, and this in turn defines $F_n$. 

We shall take advantage of the fact that the functions $F_0$, $H(F_0)$, and then $\Lambda^{-1}(F_0)F_0'$ are smooth (actually, analytic) functions of $w := z^\alpha$ on the positive real axis $\mathbb{R}^+$. First, we write down an explicit formula for the inverse $L^{-1}$ in terms of the new variable $w$. Then, we will simply estimate the functions $F_1, F_2,$ and so on in $H^3(\mathbb{R}^+)$ with respect to $w$. From now on, let us use tildes to denote functions of $w$: given a function $f$ on $\mathbb{R}^+$, we set $\tilde{f}(w) := f(w^{1/\alpha})$. 

In particular, we write \begin{equation*}
\begin{split}
\tilde{F}_0 (w) = \frac{\sin\left(\frac{\alpha\pi}{2}\right)w}{1+ 2 \cos\left(\frac{\alpha\pi}{2}\right) w + w^2 },\qquad \tilde{H(F_0)} (w) = - \frac{1+ \cos\left( \frac{\alpha\pi}{2} \right)w}{1 + 2 \cos\left(\frac{\alpha\pi}{2}\right) w + w^2}
\end{split}
\end{equation*} and the linear system may be written as \begin{equation*}
\begin{split}
\tilde{L}\tilde{f} := \tilde{f} + w\tilde{f}' + 2 \tilde{H(F_0)} \cdot \tilde{f} + 2 \tilde{F_0} \cdot \tilde{H(f)} = \tilde{g}. 
\end{split}
\end{equation*} Taking the Hilbert transform of both sides, and using complex notation \begin{equation*}
\begin{split}
V(w) &= \tilde{F_0}(w) + i \tilde{H(F_0)}(w) = \frac{1}{i+C^{(\alpha)}w},
\end{split}
\end{equation*}
and 
 \begin{equation*}
\begin{split}
U(w) &= \tilde{f}(w) + i \tilde{H(f)}(w), \qquad G(w) = \tilde{g}(w) + i \tilde{H(g)}(w), 
\end{split}
\end{equation*} the linear system takes the form \begin{equation}\label{eq:linear_alpha_complex}
\begin{split}
U + wU' - 2iVU = G.
\end{split}
\end{equation} We view the variable $w$ as varying on the strip \begin{equation*}
\begin{split}
S^{(\alpha)} := \{ w \in \mathbb{C} : 0 \le  \arg(w) \le \frac{\pi}{\alpha}  \},
\end{split}
\end{equation*} as the right hand side $G$ as well as $V$ are holomorphic on $S^{(\alpha)}$ with continuous extension up to the boundary. In this section, for simplicity we use the word ``holomorphic'' to describe such functions.

\begin{lemma}
	Consider the differential equation \eqref{eq:linear_alpha_complex} on the sector $S^{(\alpha)}$, where $G = G(w)$ is holomorphic. Then, we have the following statements: \begin{itemize}
		\item There exists a holomorphic solution $U$ if and only if $G$ satisfies \begin{equation}\label{eq:consistency_alpha}
		\begin{split}
		G'(0) -2i \left( \sin\left(\frac{\alpha\pi}{2}\right) + i \cos\left(\frac{\alpha\pi}{2}\right) \right)  G(0) = 0.
		\end{split}
		\end{equation} Any such solution satisfies $U(0) = -G(0)$.
		\item There is a unique such $U$ once we require in addition that $U'(0) = 0$. 
		\item This unique solution is given explicitly by the integral formula \begin{equation}\label{eq:explicit_alpha}
		\begin{split}
		U(w) - U(0) & = \frac{w}{(1 + 2w \cos\left(\frac{\alpha\pi}{2}\right) + w^2 )^2} \cdot \left( -(1+w^2)\cos\left(\frac{\alpha\pi}{2}\right) - 2w + i (1-w^2)\sin\left(\frac{\alpha\pi}{2}\right)  \right) \\ &\qquad\times \left( \int_0^w \left( -\frac{1+s^2}{s}\cos\left(\frac{\alpha\pi}{2}\right) - 2 - i \frac{1-s^2}{s} \sin\left(\frac{\alpha\pi}{2}\right)  \right) \frac{\hat{G}(s)}{s} ds \right),
		\end{split}
		\end{equation} where \begin{equation*}
		\begin{split}
		\hat{G}(w):= G(w) -  \frac{1+ C^{(\alpha)} iw }{1 - C^{(\alpha)} iw }G(0) ,\qquad C^{(\alpha)} = \sin\left(\frac{\alpha\pi}{2}\right) + i \cos\left(\frac{\alpha\pi}{2}\right).
		\end{split}
		\end{equation*}
	\end{itemize}
\end{lemma}

\begin{proof}
	The proof is strictly analogous to the corresponding lemma from the smooth case. 
	
	To begin with, the consistency condition \eqref{eq:consistency_alpha} simply implies that the holomorphic function $\hat{G}(s)/s$ vanishes at $s = 0$, therefore canceling the singularity $1/s$ inside the integral in \eqref{eq:explicit_alpha}. It is then clear that the formula \eqref{eq:explicit_alpha} defines a holomorphic function of $w$, and a straightforward computation shows that it provides a solution to the equation \eqref{eq:linear_alpha_complex}. 
	
	This explicit formula is easy to derive: first re-write \eqref{eq:linear_alpha_complex} as \begin{equation*}
	\begin{split}
	(1 - 2 iV) U + w U' = G,
	\end{split}
	\end{equation*} and subtracting the value at 0, we obtain \begin{equation*}
	\begin{split}
	-\frac{1}{w} \cdot  \frac{1+ C^{(\alpha)} iw }{1 - C^{(\alpha)} iw } (U - U(0)) +  (U- U(0))' = \frac{1}{w} \left( G(w) -  \frac{1+ C^{(\alpha)} iw }{1 - C^{(\alpha)} iw }G(0) \right) =: \frac{\hat{G}(w)}{w}. 
	\end{split}
	\end{equation*} The integrating factor can be computed explicitly: \begin{equation*}
	\begin{split}
	&\exp \left( \int^w \frac{1}{s} \cdot  \frac{1+ C^{(\alpha)} is }{1 - C^{(\alpha)} is }  ds  \right) \\ 
	&\qquad\qquad = \frac{w}{(1 + 2w \cos\left(\frac{\alpha\pi}{2}\right) + w^2 )^2} \cdot \left( -(1+w^2)\cos\left(\frac{\alpha\pi}{2}\right) - 2w + i (1-w^2)\sin\left(\frac{\alpha\pi}{2}\right)  \right),
	\end{split}
	\end{equation*} with the inverse \begin{equation*}
	\begin{split}
	\exp \left( - \int^w \frac{1}{s} \cdot  \frac{1+ C^{(\alpha)} is }{1 - C^{(\alpha)} is }  ds  \right) = -\frac{1+w^2}{w} \cdot \cos\left(\frac{\alpha\pi}{2}\right) - 2 - i \frac{1-w^2}{w} \cdot \sin\left(\frac{\alpha\pi}{2}\right).
	\end{split}
	\end{equation*} 
	
	This establishes \eqref{eq:explicit_alpha}, together with the fact that there is a one (complex) dimensional kernel spanned by the function  \begin{equation*}
	\begin{split}
	\frac{w}{(1 + 2w \cos\left(\frac{\alpha\pi}{2}\right) + w^2 )^2} \cdot \left( -(1+w^2)\cos\left(\frac{\alpha\pi}{2}\right) - 2w + i (1-w^2)\sin\left(\frac{\alpha\pi}{2}\right)  \right).
	\end{split}
	\end{equation*} Noting that any complex constant multiple of the above function has non-vanishing derivative at the origin, the uniqueness statement follows. Alternatively, one can simply use Taylor expansion around $w = 0$ directly in \eqref{eq:linear_alpha_complex} to deduce uniqueness.
	
	Lastly, the statement $G(0) = -U(0)$ simply follows from evaluating \eqref{eq:linear_alpha_complex} at $ w = 0$. Next, subtracting the values at 0 from both sides of \eqref{eq:linear_alpha_complex}, dividing both sides by $1/w$, and then taking the limit $w \rightarrow 0$, one sees that the condition \eqref{eq:consistency_alpha} is necessary to have a smooth solution. 
	
	This finishes the proof. \end{proof}

\begin{remark}
	Assuming further that $\Re (G(0)) = 0$ and the function $G(z^\alpha)$ is odd in $z$ when restricted onto the real axis, we have seen in Lemma \ref{lem:consistency_alpha} that the condition \eqref{eq:consistency_alpha} reduces to just having  \begin{equation}\label{eq:consistency_alpha2}
	\begin{split}
	\Re (G'(0)) + 2 \sin\left( \frac{\alpha\pi}{2} \right) \Im (G(0)) = 0.
	\end{split}
	\end{equation} Note that in the case $\alpha = 1$, this simply reduces back to $\Re(G'(0)) + 2\Im(G(0)) = 0$. 
\end{remark}

\subsubsection{Estimates for the inverse}

In this subsection, we shall restrict the variable $w$ onto the positive real axis $\mathbb{R}^+  $ and obtain precise norm estimates for the real and imaginary parts of the solution of the linear operator \eqref{eq:linear_alpha_complex} given in \eqref{eq:explicit_alpha}. 

%definition of the transformations

It will be convenient to define for each $ 0 < \alpha \le 1$ the transformations $\tilde{H}$ and $\tilde{\Lambda}^{-1}$ for functions defined on $\mathbb{R}^+$ in a way that \begin{equation*}
\begin{split}
\tilde{H}(\tilde{f})(w) := (Hf)(w^{1/\alpha}),\qquad (\tilde{\Lambda}^{-1}\tilde{f})(w) := (\Lambda^{-1}f)(w^{1/\alpha}).
\end{split}
\end{equation*} As an example, we have \begin{equation*}
\begin{split}
\tilde{H} \left( \frac{\sin\left(\frac{\alpha\pi}{2}\right)w}{1+ 2 \cos\left(\frac{\alpha\pi}{2}\right) w + w^2 } \right) = - \frac{1+ \cos\left( \frac{\alpha\pi}{2} \right)w}{1 + 2 \cos\left(\frac{\alpha\pi}{2}\right) w + w^2}.
\end{split}
\end{equation*} Let us write out an explicit integral representation for the operator $\tilde{H}$. From the definition, \begin{equation*}
\begin{split}
\tilde{H}(\tilde{f})(w) &= \frac{1}{\pi}\,p.v.\int_{\mathbb{R}^+} \frac{2tf(t)}{w^{2/\alpha}-t^2} dt = \frac{1}{\pi}\,p.v.\int_{\mathbb{R}^+} \frac{2t\tilde{f}(t^{\alpha})}{w^{2/\alpha}-t^2} dt \\
&= \frac{1}{\pi}\,p.v.\int_{\mathbb{R}^+} \frac{2s^{2/\alpha-1} }{w^{2/\alpha} - s^{2/\alpha}}\tilde{f}(s) ds .
\end{split}
\end{equation*} In the case $\alpha = 1/n$ for some integer $n \ge 1$, notice that the kernel is a rational function of $w$ and $s$. From this it is easy to obtain an $L^2$ estimate for the operator $\tilde{H} = \tilde{H}^{(n)}$, see Lemma \ref{lem:L2_hilbert_type} in the Appendix. From now on, we restrict ourselves to such values of $\alpha$. Next, note that the operator $\tilde{\Lambda}^{-1}$ has the representation \begin{equation*}
\begin{split}
(\tilde{\Lambda}^{-1}\tilde{f})(w) &:=  \Lambda^{-1}f(w^{1/\alpha}) = \int_0^{w^{1/\alpha}} Hf(s)ds = \int_0^w \tilde{H}\tilde{f}(t) \cdot  \frac{1}{\alpha}t^{\frac{1-\alpha}{\alpha}}dt.
\end{split}
\end{equation*}

Before we state the necessary $H^3$ estimates for the inverse $\tilde{L}^{-1}$, let us first write out the real part of $U(w)$ from \eqref{eq:explicit_alpha}. Denoting the real part by $\tilde{L}^{-1}(\tilde{g})$, we have: \begin{equation*}
\begin{split}
\tilde{L}^{-1}(\tilde{g}) = \tilde{L}^{-1,I}(\tilde{g}) + \tilde{L}^{-1,II}(\tilde{g})
\end{split}
\end{equation*} with  \begin{equation*}
\begin{split}
&\tilde{L}^{-1,I}(\tilde{g}) := \frac{w \left( -(1+w^2) \cos\left( \frac{\alpha\pi}{2}\right) - 2w \right) }{\left(1 + 2 \cos\left( \frac{\alpha\pi}{2}\right)w + w^2 \right)^2} \\
&\qquad\qquad \times \left[
\int_0^w \left( - \frac{1+s^2}{s} \cos\left( \frac{\alpha\pi}{2} \right)  - 2 \right) \cdot \frac{1}{s} \left( \tilde{g}(s) - \frac{s}{ 1 + 2 \cos\left( \frac{\alpha\pi}{2}\right)s + s^2  } \tilde{g}'(0) \right) ds \right. \\
& \left. \qquad\qquad\qquad + \int_0^w   \left( \frac{1-s^2}{s} \sin\left(\frac{\alpha\pi}{2}\right) \right) \cdot \frac{1}{s} \left( \tilde{H}\tilde{g}(s) - \frac{1-s^2}{1 + 2 \cos\left( \frac{\alpha\pi}{2}\right)s + s^2 }  \tilde{H}\tilde{g}(0)   \right) ds  \right],
\end{split}
\end{equation*} and \begin{equation*}
\begin{split}
&\tilde{L}^{-1,II}(\tilde{g}) := \frac{w  (1-w^2)\sin\left(\frac{\alpha\pi}{2}\right)   }{\left(1 + 2 \cos\left( \frac{\alpha\pi}{2}\right)w + w^2 \right)^2} \\
&\qquad\qquad \times \left[
\int_0^w \left(  \frac{1+s^2}{s} \cos\left( \frac{\alpha\pi}{2} \right) +   2 \right) \cdot \frac{1}{s}\left( \tilde{H}\tilde{g}(s) - \frac{1-s^2}{1 + 2 \cos\left( \frac{\alpha\pi}{2}\right)s + s^2 }  \tilde{H}\tilde{g}(0)   \right) \right.  ds \\
&\left. \qquad\qquad\qquad + \int_0^w   \left( \frac{1-s^2}{s} \sin\left(\frac{\alpha\pi}{2}\right) \right) \cdot \frac{1}{s}\left( \tilde{g}(s) - \frac{s}{ 1 + 2 \cos\left( \frac{\alpha\pi}{2}\right)s + s^2  } \tilde{g}'(0) \right)  ds  \right],
\end{split}
\end{equation*} where we have used the consistency condition \eqref{eq:consistency_alpha2}. 

We now split the above inverse operator into several pieces: we define for $\sigma \in \{ -1 , 0 , 1\}$
\begin{equation}\label{eq:T_alpha}
\begin{split}
T^{\sigma}(\tilde{g})(w) &:= P(w) \cdot \int_0^w s^\sigma \left( \frac{\tilde{g}(s)}{s} - \frac{1}{ 1 + 2 \cos\left( \frac{\alpha\pi}{2}\right)s + s^2  }\tilde{g}'(0) \right) ds
\end{split}
\end{equation} and \begin{equation}\label{eq:S_alpha}
\begin{split}
S^{\sigma}(\tilde{g})(w) &:= P(w) \cdot \int_0^w s^\sigma \left( \frac{\tilde{H}\tilde{g}(s)}{s} -  \frac{1-s^2}{s(1 + 2 \cos\left( \frac{\alpha\pi}{2}\right)s + s^2) }  \tilde{H}\tilde{g}(0) \right) ds \\
&= P(w) \cdot \int_0^w s^\sigma \left( \frac{\tilde{H}\tilde{g}(s)-\tilde{H}\tilde{g}(0)}{s} +  \frac{2\cos\left( \frac{\alpha\pi}{2}\right) + 2s}{1 + 2 \cos\left( \frac{\alpha\pi}{2}\right)s + s^2 }  \tilde{H}\tilde{g}(0) \right) ds
\end{split}
\end{equation} where the prefactor $P(w)$ equals \begin{equation*}
\begin{split}
P(w) = w \cdot \frac{ \left( -(1+w^2) \cos\left( \frac{\alpha\pi}{2}\right) - 2w \right) }{\left(1 + 2 \cos\left( \frac{\alpha\pi}{2}\right)w + w^2 \right)^2}.
\end{split}
\end{equation*}

\begin{lemma}\label{lem:keyest_alpha}
	Assume that $\tilde{E}$ and $\tilde{F}$ vanish at zero. Then we have the $H^3$ estimates of the form  \begin{equation}\label{eq:keyest_alpha}
	\begin{split}
	\V T^\sigma(\tilde{F}) \V_{H^3} + \alpha \cdot \V S^\sigma(\tilde{F}) \V_{H^3}  &\le c \V \tilde{F} \V_{H^3},
	\end{split}
	\end{equation}\begin{equation}\label{eq:keyest_alpha_D}
	\begin{split}
	\V T^\sigma(w\tilde{F}') \V_{H^3} + \alpha \cdot \V S^\sigma(w\tilde{F}') \V_{H^3}  &\le c \V \tilde{F} \V_{H^3},
	\end{split}
	\end{equation} and % \tilde{\Lambda}^{-1}(\tilde{E}) \cdot    \tilde{F}'
	\begin{equation}\label{eq:keyest_alpha_HD}
	\begin{split}
	\V T^\sigma\left( (\Lambda^{-1}E \cdot F')^{\widetilde{ }} \right) \V_{H^3} + \alpha \cdot \V S^\sigma \left( (\Lambda^{-1}E \cdot F')^{\widetilde{ }} \right)\V_{H^3}  &\le c \V \tilde{E} \V_{H^3} \cdot \V \tilde{F} \V_{H^3} ,
	\end{split}
	\end{equation} for each $\sigma \in \{ -1, 0 ,1 \}$, where $c > 0$ is some absolute constant independent on $0 < \alpha \le 1$. 
\end{lemma}

\begin{proof}
	The proof is completely analogous to those of the estimates in Lemmas \ref{lem:L2} and \ref{lem:H3}. 
	
	\medskip
	
	\textbf{(i) proof of \eqref{eq:keyest_alpha}}
	
	\medskip
	
	We begin with the proof of the inequality \eqref{eq:keyest_alpha}. Consider first the function $T^\sigma(\tilde{F})$ in the most difficult case $\sigma = -1$: after a simple rewriting, we have that \begin{equation*}
	\begin{split}
	T^{-1}(\tilde{F})(w) = P(w) \cdot \int_0^w \frac{1}{s}\left( \frac{\tilde{F}(s)}{s} - \tilde{F}'(0) \right) + \frac{2 \cos\left(\frac{\alpha\pi}{2}\right) + s}{1 + 2\cos \left(\frac{\alpha\pi}{2}\right)s + s^2} \tilde{F}'(0) ds.
	\end{split}
	\end{equation*}  We first deal with the second term: \begin{equation*}
	\begin{split}
	P(w) \tilde{F}'(0) \cdot \int_0^w \frac{2 \cos\left(\frac{\alpha\pi}{2}\right) + s}{1 + 2\cos \left(\frac{\alpha\pi}{2}\right)s + s^2} ds.
	\end{split}
	\end{equation*} In $L^2$, this is bounded by \begin{equation*}
	\begin{split}
	C \V \tilde{F}\V_{H^2} \cdot \V P(w) \cdot \int_0^w \frac{2 \cos\left(\frac{\alpha\pi}{2}\right) + s}{1 + 2\cos \left(\frac{\alpha\pi}{2}\right)s + s^2} ds\V_{L^\infty} \le 	C \V \tilde{F}\V_{H^2} \cdot \V P(w) \ln(1 + w) \V_{L^\infty} \le 	C \V \tilde{F}\V_{H^2},
	\end{split}
	\end{equation*} and it is straightforward to show that this is bounded in $H^3$ again by a constant multiple of $\V \tilde{F}\V_{H^2}$. Next, the first term of $T^{-1}(\tilde{F})$ is bounded in $L^2$ by \begin{equation*}
	\begin{split}
	\V wP(w)\V_{L^\infty} \cdot \V \frac{1}{w}\int_0^w \frac{1}{s}\left( \frac{\tilde{F}(s)}{s} - \tilde{F}'(0) \right) ds \V_{L^2} \le C \V \tilde{F} \V_{H^2}
	\end{split}
	\end{equation*} just as in the smooth case, using the Hardy inequality. To obtain the $H^3$ bound, we may assume that a derivative falls on the integral to get \begin{equation*}
	\begin{split}
	\frac{P(w)}{w} \cdot \left( \frac{\tilde{F}(w)}{w} - \tilde{F}'(0)  \right),
	\end{split}
	\end{equation*} which is bounded in $H^2$ by a constant multiple of $\V \tilde{F}\V_{H^3}$ again using the Hardy inequality. 
	
	The cases $\sigma = 0, 1$ are only simpler: we can simply decompose $T^\sigma(\tilde{F})$ into \begin{equation*}
	\begin{split}
	\left( P(w)\int_0^w s^{\sigma - 1} \tilde{F}(s) ds \right) - \left( P(w) \tilde{F}'(0) \cdot \int_0^w \frac{s^\sigma}{1+ 2 \cos \left(\frac{\alpha\pi}{2}\right) s + s^2} ds \right),
	\end{split}
	\end{equation*} and then clearly both terms can be bonded in $H^3$ by $C \V \tilde{F}\V_{H^3}$.  
	
	The corresponding estimate for the function $S^\sigma(\tilde{F})$ can be carried out in a similar way. Let us only point out that we lose by a factor of $1/\alpha$ simply because of the loss in $H^3$-estimate from Lemma \ref{lem:L2_hilbert_type}: \begin{equation*}
	\begin{split}
	\V \tilde{H}\tilde{F} \V_{H^3} \le \frac{C}{\alpha} \V \tilde{F}\V_{H^3}.
	\end{split}
	\end{equation*}
	
	\medskip
	
	\textbf{(ii) proof of \eqref{eq:keyest_alpha_D}}
	
	\medskip
	
	We proceed to the proof of \eqref{eq:keyest_alpha_D}, starting with the case of $T^{-1}(w\tilde{F}')$. We need to estimate: \begin{equation*}
	\begin{split}
	&P(w) \int_0^w \frac{1}{s} \left( \tilde{F}'(s) - \frac{1}{1+2 \cos \left(\frac{\alpha\pi}{2}\right) s + s^2} \tilde{F}'(0)  \right) ds \\
	& \qquad = P(w) \cdot \left( \int_0^w \frac{\tilde{F}'(s) - \tilde{F}'(0)}{s} ds + \tilde{F}'(0) \cdot \int_0^w \frac{2\cos \left(\frac{\alpha\pi}{2}\right) +s}{1+ 2\cos \left(\frac{\alpha\pi}{2}\right) s + s^2} ds      \right),
	\end{split}
	\end{equation*}  but the corresponding proof from the smooth case carries over to this one, which bounds the above in $H^3$ by a constant multiple of $\tilde{F}$ in $H^3$. The cases $\sigma = 0, 1$ are simpler and can be done as in the case of \eqref{eq:keyest_alpha}. 
	
	Given the desired bound for $T^{\sigma}(w\tilde{F}')$, the corresponding result for $S^\sigma(w\tilde{F}')$ similarly follows, using the convenient fact that \begin{equation*}
	\begin{split}
	\tilde{H}\left( w \tilde{F}' \right) = w (\tilde{H}(\tilde{F}))',
	\end{split}
	\end{equation*} see \eqref{eq:identity} in the Appendix. 
	
	\medskip
	
	\textbf{(iii) proof of \eqref{eq:keyest_alpha_HD}}
	
	\medskip
	
	To begin with, it is necessary to note that \begin{equation*}
	\begin{split}
	(\Lambda^{-1}E \cdot F')^{\widetilde{ }}(w)  = (\frac{1}{z}(\Lambda^{-1}E)(z) \cdot zF'(z))^{\widetilde{ }}(w) = w^{-\frac{1}{\alpha}}\tilde{\Lambda}^{-1} (\tilde{E})(w) \cdot \alpha w\tilde{F}'(w)
	\end{split}
	\end{equation*} and that its derivative at $w = 0$ equals \begin{equation*}
	\begin{split}
	\alpha \cdot \tilde{H}(\tilde{E})(0) \cdot \tilde{F}'(0). 
	\end{split}
	\end{equation*}
	
	Let us restrict ourselves to obtaining an $H^3$ estimate for $T^{-1}((\Lambda^{-1}E \cdot F')^{\widetilde{ }})$. We need to bound \begin{equation*}
	\begin{split}
	T^{-1}(\tilde{\Lambda}^{-1}(\tilde{E})\cdot \tilde{F}')(w) &= P(w) \cdot \alpha\int_0^w \frac{1}{s}\left( s^{-\frac{1}{\alpha}}\tilde{\Lambda}^{-1}(\tilde{E})(s) \cdot \tilde{F}'(s) - \tilde{H}(\tilde{E})(0)\cdot \tilde{F}'(0) \right) ds  \\
	&\qquad + P(w)\cdot \alpha\tilde{H}(\tilde{E})(0)\cdot \tilde{F}'(0)\cdot \int_0^w \frac{2 \cos\left(\frac{\alpha\pi}{2}\right) + s}{1 + 2\cos \left(\frac{\alpha\pi}{2}\right)s + s^2}  ds.
	\end{split}
	\end{equation*} To estimate the latter term, it suffices to observe that \begin{equation*}
	\begin{split}
	\tilde{H}\left(  \Lambda^{-1}(E) \cdot F'  \right)^{\tilde{}}(0) &= H\left(\Lambda^{-1}(E) \cdot F'\right)(0) \\
	&=-\frac{2}{\pi}\int_0^\infty \frac{1}{z} \left( \int_0^z H(E)(s)ds\right) F'(z) dz \\
	&= -\frac{2}{\pi}\int_0^\infty \left( w^{-\frac{1}{\alpha}} \int_0^w \tilde{H}(\tilde{E})(t) \cdot \frac{1}{\alpha}t^{\frac{1-\alpha}{\alpha}} dt \right) \cdot \tilde{F}'(w) dw 
	\end{split}
	\end{equation*} (after a change of variable) so that in particular,  \begin{equation}\label{eq:Linfty_alpha}
	\begin{split}
	\alpha\cdot \left|\tilde{H}\left(  \Lambda^{-1}(E) \cdot F'  \right)^{\tilde{}}(0)\right|\le C\alpha \V \tilde{H}(\tilde{E}) \V_{L^2} \cdot \V \tilde{F}\V_{H^1} \le C \V \tilde{E}\V_{H^3}\cdot\V \tilde{F}\V_{H^3}
	\end{split}
	\end{equation} using the Hardy-type inequality \eqref{eq:Hardy_alpha}. The former term can be estimated in $H^3$ by following along the lines of the corresponding proof from Lemmas \ref{lem:L2} and \ref{lem:H3}, using the Hardy-type inequality \eqref{eq:Hardy_alpha} instead of the usual Hardy inequality.
	
	Regarding $S^{-1}((\Lambda^{-1}E \cdot F')^{\widetilde{ }})$, one again just needs to follow along the arguments given in the smooth case. The only difference is that one needs first to decompose the operator $\tilde{H}$ into $\sum_{j=-n+1}^{n-1} \tilde{H}_j$ where $n = 1/\alpha$ (see \eqref{eq:piece0},\eqref{eq:pieces} in the Appendix for the definition of the pieces $\tilde{H}_j$), and then use the identity \eqref{eq:div_by_w} which is a convenient substitute for \eqref{eq:Hilbert_divid_z}.
\end{proof}

Given the above estimates, let us conclude the proof of Theorem \ref{thm:main2}. 

\begin{proof}[Proof of Theorem \ref{thm:main2}]
	Let us first proceed to show that, with some absolute constants $c > 0$ and $c_0 > 0$ uniform over $0 < \alpha \le 1$, we have the bounds \begin{equation*}
	\begin{split}
	\V \tilde{F}_n \V_{H^3} \le c\alpha (c_0\alpha)^{n}, \qquad |\lambda_n| \le c (c_0 \alpha)^n.
	\end{split}
	\end{equation*}  for all $n \ge 0$. Note that in the case $n = 0$, these estimates trivially hold.
	
	We set $ \V \tilde{F}_n \V_{H^3} =: \alpha^{n+1} \mu_n$ , $ | \lambda_n | =: \alpha^n l_n $ and let us write down the set of inequalities for the sequences $\{ \mu_n \}$ and $\{ l_n \}$. 
	
	To begin with, we recall that in terms of the original variable $z$, \begin{equation*}
	\begin{split}
	L(F_n) = G_n - \frac{1}{\alpha} \lambda_n z F_0', 
	\end{split}
	\end{equation*} with \begin{equation*}
	\begin{split}
	G_n(z) = \sum_{j=0}^{n-1} \Lambda^{-1}(F_j) \cdot F'_{n-1-j} - \frac{1}{\alpha}\sum_{j=1}^{n-1} \lambda_j zF_{n-j}' - 2 \sum_{j=1}^{n-1} H(F_{n-j}) \cdot F_j
	\end{split}
	\end{equation*} and \begin{equation*}
	\begin{split}
	H(G_n)(z) = \sum_{j=0}^{n-1} H\left(\Lambda^{-1}(F_j) \cdot F'_{n-1-j}\right) - \frac{1}{\alpha}\sum_{j=1}^{n-1} \lambda_j zH(F_{n-j})' +  \sum_{j=1}^{n-1} \left( F_{n-j} \cdot F_j - H(F_{n-j}) \cdot H(F_j) \right).
	\end{split}
	\end{equation*} Also recall that $\lambda_n$ is defined by \begin{equation*}
	\begin{split}
	\lambda_n = \frac{1}{\sin\left( \frac{\alpha\pi}{2} \right)} \tilde{G}_n'(0) + 2 \tilde{H}(\tilde{G}_n)(0). 
	\end{split}
	\end{equation*}
	
	A straightforward computation shows that $\tilde{G}_n'(0)$ equals \begin{equation*}
	\begin{split}
	\sum_{j=0}^{n-1} \tilde{H}(\tilde{F}_j)(0) \cdot \alpha \tilde{F}'_{n-1-j}(0) - \sum_{j=1}^{n-1}\lambda_j \tilde{F}_{n-j}'(0) - 2\sum_{j=1}^{n-1} \tilde{H}(\tilde{F}_{n-j})(0) \cdot \tilde{F}_{n-j}'(0).
	\end{split}
	\end{equation*} Next, recall that we had \begin{equation*}
	\begin{split}
	\left| \tilde{H}\left(  \Lambda^{-1}(F_j) \cdot F'_{n-1-j}  \right)^{\tilde{}}(0) \right|\le C \V \tilde{H}(\tilde{F}_j) \V_{L^2} \cdot \V \tilde{F}_{n-1-j}\V_{H^1}
	\end{split}
	\end{equation*} from \eqref{eq:Linfty_alpha}. Since $\tilde{H}(\tilde{G}_n)(0)$ equals \begin{equation*}
	\begin{split}
	\sum_{j=0}^{n-1} \tilde{H}\left(  \Lambda^{-1}(F_j) \cdot F'_{n-1-j}  \right)^{\tilde{}}(0) - \sum_{j=1}^{n-1} \tilde{H}(\tilde{F}_{n-j})(0) \cdot \tilde{H}(\tilde{F}_j)(0),
	\end{split}
	\end{equation*} we finally deduce that \begin{equation*}
	\begin{split}
	|\lambda_n| &\le \frac{C}{\alpha} \left( \sum_{j=0}^{n-1} \mu_j \alpha^j \cdot \alpha \cdot \mu_{n-1-j} \alpha^{n-j}  + \sum_{j=1}^{n-1} l_j \alpha^j \cdot \mu_{n-j} \alpha^{n-j+1} + \sum_{j=1}^{n-1} \mu_{n-j} \alpha^{n-j} \cdot \mu_j \alpha^{j+1} \right) \\
	&\qquad + C \left( \sum_{j=0}^{n-1} \mu_j \alpha^j \cdot \mu_{n-1-j} \alpha^{n-j}+ \sum_{j=1}^{n-1} \mu_{n-j} \alpha^{n-j} \cdot \mu_j \alpha^j \right),
	\end{split}
	\end{equation*} using the $H^3(\mathbb{R}) \subset L^\infty(\mathbb{R}) $ embedding together with the estimate \begin{equation*}
	\begin{split}
	\V \tilde{H}(\tilde{F_j})\V_{H^3} \le \frac{C}{\alpha} \V \tilde{F}_j\V_{H^3}. 
	\end{split}
	\end{equation*} Equivalently, \begin{equation}\label{eq:l_est_final}
	\begin{split}
	l_n \le C\left( \mu_0\mu_{n-1}  + \sum_{j=1}^{n-1} \left( \mu_j \mu_{n-1-j} + l_j \mu_{n-j} + \mu_{n-j}\mu_j \right)   \right)
	\end{split}
	\end{equation} with some absolute constant $C > 0$ uniform on $0 < \alpha \le 1$. 
	
	Now, given the value of $\lambda_n$, \begin{equation*}
	\begin{split}
	\tilde{F}_n:=\tilde{L}^{-1}\left( \tilde{G}_n - \lambda_n w \tilde{F}_0' \right)
	\end{split}
	\end{equation*} can be written as a linear combination of $T^\sigma$, $S^{\sigma'}$ with $\sigma,\sigma' \in \{ -1, 0 , 1\}$, and it is important to notice that whenever we use the operator $S^{\sigma'}$, its coefficient in the expansion of $\tilde{L}^{-1}$ comes with a factor of $\alpha$. Then, applying Lemma \ref{lem:keyest_alpha} to each term in $\tilde{G}_n$, we deduce that \begin{equation*}
	\begin{split}
	\mu_n \alpha^{n+1} \le C \left(  l_n \alpha^n \cdot \alpha  \right) + C \alpha^{n+1} \sum_{j=1}^{n-1} \left( \mu_j \mu_{n-1-j} + l_j \mu_{n-j} + \mu_j \mu_{n-j} \right),
	\end{split}
	\end{equation*} or equivalently, \begin{equation*}
	\begin{split}
	\mu_n \le C \left( l_n + \mu_0 \mu_{n-1} +  \sum_{j=1}^{n-1} \left( \mu_j \mu_{n-1-j} + l_j \mu_{n-j} + \mu_j \mu_{n-j} \right)  \right).
	\end{split}
	\end{equation*} Now that the form of the sequence of inequalities for $\{ l_n\}, \{ \mu_n \}$ are equivalent with that of $\{ |\lambda_n| \}, \{ \mu_n \}$ from the smooth case, we can run the exact same argument to deduce the desired bounds. 
	
	This shows that the series \begin{equation*}
	\begin{split}
	F^{(\alpha)}(z) &= F^{(\alpha)}_0(z) + \sum_{n=1}^\infty a^n F_n^{(\alpha)}(z),\qquad
	\lambda^{(\alpha)}(a) = \sum_{n=1}^\infty a^n \lambda_n^{(\alpha)}
	\end{split}
	\end{equation*} are convergent for some interval $a \in (-1/(c_0\alpha),1/(c_0\alpha))$ on which it can be also guaranteed that $\lambda^{(\alpha)}(a) > -1$. 
	
	Next, it is easy to see that the self-similar profile $F^{(\alpha)}$ indeed belongs to $C^\alpha(\mathbb{R})$: it suffices to observe that $\tilde{F}^{(\alpha)}$ belongs to $\dot{C}^1 \cap L^\infty $. To show decay of the function $F^{(\alpha)}(z)$, we simply argue using the equation \eqref{eq:DeG_resc_alpha} just as we did in the proof of Theorem \ref{thm:main1prime}. Indeed, note that if $F$ solves $$F+\frac{1+\lambda}{\alpha} zF'+ 2FH(F)+a\Lambda^{-1}(F)  F'=0$$ and if $F=\tilde{F}(z^\alpha)$ with $\tilde{F}\in H^3$ (as we have shown above) then $\tilde{F}$ solves 
	$$ \tilde{F}(w) + (1+\lambda)w \tilde{F}'(w) + 2\tilde{F}\cdot \tilde{H}(\tilde{F}) + a w^{-\frac{1}{\alpha}} \tilde{\Lambda}^{-1}(\tilde{F}) \cdot w\tilde{F}'(w) = 0.
	$$ Proceeding as in the proof of Theorem \ref{thm:main1prime}, we obtain that $\tilde{F}$ decays like $w^{-\frac{1}{1+\lambda}}$, and therefore $F$ decays like $z^{-\frac{\alpha}{1+\lambda}}.$ Now we are done.

\end{proof}

%	Next, from the equation \eqref{eq:DeG_resc_alpha} for $F^{(\alpha)}$, it is direct to see that the function $zF^{(\alpha)'}(z)$ belongs to $L^\infty$.\footnote{It is actually straightforward to show that $w\tilde{F}^{(\alpha)'}(z)$ also belongs to $H^3$ (and actually in $H^m$ for any $m \ge 0$). To do this, one would write down an explicit integral formula for $w\tilde{F}_n^{(\alpha)'}$ and then establish an appropriate $H^3$-bound along the iteration sequence.} 

\section{A method to get singularity formation in the full range}\label{sec:circle}

To close the paper, we would like to mention another method which likely could lead to finite-time singularity formation for $C^\alpha$ solutions to the De Gregorio model which are periodic. Consider De Gregorio's model:
$$\partial_t \omega+2u\partial_x\omega=2\partial_x u\omega.$$
Assume that the data is odd in $x$, $2\pi$ periodic, and positive in $(0,\pi)$. Suppose further that $$|\omega_0(x)|\geq C (\sin(x))^\alpha$$ on $(0,\pi)$ for some $\alpha<1$ and some $C>0$. Now define $$f_0(x):=\int_0^x \frac{dy}{\omega_0(y)}.$$ Since $\omega_0$ is bounded from below as above and since $\alpha<1$, $f_0$ is well-defined. In general, if we define $$f(t,x)=\int_{0}^x \frac{dy}{\omega(t,y)}$$ we see that $f$ is a $2\pi$ periodic even function satisfying a transport equation:

$$\partial_t f + 2u\partial_x f=0,$$
$$u=-\Lambda^{-1}(\frac{1}{\partial_x f}),$$  with $f_0(x)\approx|\sin(x)|^\alpha$
near $x=0$. 
To prove that $\omega$ becomes singular in finite time, it suffices simply to prove that there is a time where $f$ ceases to have a cusp at $x=0$ in the sense  that the quantity $$A(t):=\lim_{x\rightarrow 0}\frac{|f(x,t)|}{|x|^\alpha}$$ hits zero in finite time.  Since the velocity field $u$ is directed away from the origin, $A(t)$ is certainly a decreasing function. However, to prove that $A(t)$ actually hits zero in finite time, it is necessary to understand the operator $f\rightarrow\Lambda^{-1}(\frac{1}{\partial_x f})$ more precisely.  

This reformulation of the problem allows us to clearly see how the case of $\alpha<1$ is distinguished from that of $\alpha\geq 1$, since $f$ cannot even be defined from $\omega$ if $\omega$ is $C^1$ and vanishes. This leads one to conjecture that the $C^1$ case and the $C^\alpha$ case ($\alpha<1$) are actually quite different. Finally, for the Okamoto-Sakajo-Wunsch models one can define $f$ similarly, and the right definition for $f$ in that case is $$f(t,x)=\int_0^x \frac{dy}{\omega(t,y)^{a/2}}.$$ For $a<2$, $f$ can be defined even when $\omega$ is $C^\infty.$ This leads one to further conjecture that singularity formation will occur in the full range $a<2$.

\section{Acknowledgements}

The authors would like to thank Tej-Eddine Ghoul, Vu Hoang, Hao Jia, Andrew Majda, Nader Masmoudi, Huy Nguyen, Vladimir Sverak, and Vlad Vicol for helpful remarks and stimulating discussions.  T.\ M.\ Elgindi acknowledges funding from NSF grant DMS-1402357.

\appendix

\section{Appendix}

\subsection{Properties of the Hilbert transform}

We collect a few simple properties of the Hilbert transform. It will be implicitly assumed that the functions $f$ and $g$ are in the domain of the Hilbert transform. 

\begin{lemma}[The Tricomi identity]
	Given two functions $f$ and $g$, 
	\begin{equation}\label{eq:tricomi}
	\begin{split}
	H(fg) = H(f)g + fH(g) + H(Hf Hg). 
	\end{split}
	\end{equation}
\end{lemma}
\begin{proof}
	There are complex analytic functions $u$ and $v$ on the upper half plane whose restrictions on the real line equals \begin{equation*}
	\begin{split}
	u &= f + iHf,\\
	v &= g + iHg.
	\end{split}
	\end{equation*} Then, it suffices to note that \begin{equation*}
	\begin{split}
	uv = fg-HfHg + i(fHg + gHf),
	\end{split}
	\end{equation*} since then \begin{equation*}
	\begin{split}
	H(fg - HfHg) = fHg + gHf
	\end{split}
	\end{equation*} holds.
\end{proof}

The following identities are very well-known:
\begin{lemma}
	Assuming that $zf(z) \in L^2(\mathbb{R})$, the Hilbert transform of $zf(z)$ is related to the Hilbert transform of $f$ via \begin{equation}\label{eq:Hilbert_multip_z}
	\begin{split}
	H\left(wf(w)\right)(z) = z Hf(z) - \frac{1}{\pi} \int_{ \mathbb{R}} f(w)dw.
	\end{split}
	\end{equation} Moreover, if one assumes that $(f(z)-f(0))/z \in L^2(\mathbb{R})$, then we have \begin{equation}\label{eq:Hilbert_divid_z}
	\begin{split}
	H\left( \frac{f(w)-f(0)}{w} \right)(z) = \frac{Hf(z) - Hf(0)}{z}. 
	\end{split}
	\end{equation}
\end{lemma}

\begin{proof}
	Note that \begin{equation*}
	\begin{split}
	H\left(wf(w)\right)(z) := p.v.\, \frac{1}{\pi} \int \frac{wf(w)}{z-w} dw =  \frac{1}{\pi} \int \frac{(w-z)f(w)}{z-w} dw  + z\left( p.v.\,  \frac{1}{\pi} \int \frac{f(w)}{z-w} dw \right).
	\end{split}
	\end{equation*} The proof of the second identity is strictly analogous.
\end{proof}

We state and prove a few elementary facts regarding the transforms $H^{(\alpha)}$ of Hilbert-type which appear in Section \ref{sec:Hoelder}. We first recall the definition of these transforms. Given some positive integer $n$, and a function $f$ defined on $\mathbb{R}^+$, we consider the transformation $\tilde{H}^{(n)}$ for $w > 0$: \begin{equation}\label{eq:Hilbert_type}
\begin{split}
\tilde{H}^{(n)}(f)(w) = \frac{1}{\pi}\, p.v. \int_{\mathbb{R}^+} \frac{nt^{2n-1}}{w^{2n}- t^{2n}} f(t) dt .
\end{split}
\end{equation}

First, we may decompose the kernel as follows:

\begin{lemma}
	We have \begin{equation*}
	\begin{split}
	\tilde{H}^{(n)} (f)(w) &= \frac{1}{\pi} \sum_{j=1}^{2n}  \int_{\mathbb{R}^+} \frac{-\zeta_{2n}^j}{w- \zeta_{2n}^j t} f(t) dt \\
	&= \frac{1}{\pi}\, p.v.  \int_{\mathbb{R}^+} \frac{2t}{w^2-t^2} dt + \sum_{j=1}^{n} \frac{1}{\pi} \int_{\mathbb{R}^+} \frac{2t - 2 \Re(\zeta_{2n}^j)w }{w^2 - 2 \Re(\zeta_{2n}^j) wt + t^2 }f(t) dt 
	\end{split}
	\end{equation*} where $\zeta_{2n} := \exp\left(i \pi/n  \right)$ is the primitive $2n$-th root of unity.
\end{lemma}

\begin{proof}
	This is a purely algebraic statement which can be checked directly. 
\end{proof}

We set \begin{equation}\label{eq:piece0}
\begin{split}
\tilde{H}^{(n)}_0 (f)(w) = \frac{1}{\pi}\, p.v.  \int_{\mathbb{R}^+} \frac{2t}{w^2-t^2} f(t) dt
\end{split}
\end{equation} and for $j \in \{ \pm 1, \cdots, \pm (n-1) \}$ \begin{equation}\label{eq:pieces}
\begin{split}
\tilde{H}^{(n)}_j (f)(w) = \frac{1}{\pi} \int_{\mathbb{R}^+} \frac{-\zeta_{2n}^j}{w - \zeta_{2n}^j t }f(t) dt .
\end{split}
\end{equation} Note that $\tilde{H}^{(n)}_0 $ is exactly the Hilbert transform of the odd function $f(|w|) \mathrm{sgn}(w)$ restricted onto $\mathbb{R}^+$. In particular, \begin{equation*}
\begin{split}
\V \tilde{H}^{(n)}_0(f) \V_{L^2} \le \V f\V_{L^2}. 
\end{split}
\end{equation*}

The following lemma follows directly from the definition of the operators $\tilde{H}^{(n)}_j$. 

\begin{lemma}
	For $1 \le |j| < n$, we have the identities \begin{equation}\label{eq:mult_by_w}
	\begin{split}
	\tilde{H}^{(n)}_j \left( t f( t ) \right)(w) =  -\frac{1}{\zeta_{2n}^j} \cdot w\tilde{H}^{(n)}_j(f)(w) - \frac{1}{\pi}\int_{\mathbb{R}^+} f(t)dt,
%	\tilde{H}^{(n)}\left( t f( t ) \right)(w) = w \cdot \, \left(  \tilde{H}^{(n)}_0(f)(w) + \sum_{|j|=1}^n \zeta_{2n}^j \cdot \tilde{H}^{(n)}_j(f)(w) \right)  + \frac{2n}{\pi}\int_{\mathbb{R}^+} f(t)dt,
	\end{split}
	\end{equation} \begin{equation}\label{eq:div_by_w}
	\begin{split}
	\tilde{H}^{(n)}_j \left( \frac{f(t) - f(0)}{t} \right)(w) = \zeta_{2n}^j \cdot \frac{\tilde{H}^{(n)}_j f(w) - \tilde{H}^{(n)}_j f(0)}{w},
%	\tilde{H}^{(n)}\left( \frac{f(t) - f(0)}{t} \right)(w) = \frac{\tilde{H}^{(n)}_0 f(w) - \tilde{H}^{(n)}_0 f(0) }{w} +  \sum_{|j| = 1}^n \zeta_{2n}^j \cdot \frac{\tilde{H}^{(n)}_j f(w) - \tilde{H}^{(n)}_j f(0)}{w},
	\end{split}
	\end{equation} and  \begin{equation}\label{eq:diff}
	\begin{split}
	\tilde{H}_j^{(n)}(f')(w) = -\zeta^j_{2n} \cdot \tilde{H}^{(n)}_j(f)'.
	\end{split}
	\end{equation} As a consequence, we obtain the identity \begin{equation}
	\begin{split}
	\tilde{H}_j^{(n)}(t f'(t))(w) = w (\tilde{H}_j^{(n)}f)'(w),
	%	\tilde{H}^{(n)}(t f'(t))(w) = w (\tilde{H}^{(n)}f)'(w).
	\end{split}
	\end{equation} which in particular implies that \begin{equation}\label{eq:identity}
	\begin{split}
	\tilde{H}^{(n)}(t f'(t))(w) = w (\tilde{H}^{(n)}f)'(w).
%	\tilde{H}^{(n)}(t f'(t))(w) = w (\tilde{H}^{(n)}f)'(w).
	\end{split}
	\end{equation}
\end{lemma}

\begin{lemma}\label{lem:L2_hilbert_type}
	For each $1 \le |j| \le n$, we have the estimate \begin{equation*}
	\begin{split}
	\V \tilde{H}^{(n)}_j (f)\V_{L^2} \le C \ln\left( 1 + \frac{1}{\sin^2\left(\frac{j\pi}{n}\right)} \right)  \V f\V_{L^2}
	\end{split}
	\end{equation*} with some absolute constant $C > 0$. In particular, for each $n \ge 1$, \begin{equation}\label{eq:L2_Hilbert_alpha}
	\begin{split}
	\V \tilde{H}^{(n)}(f)\V_{L^2} \le Cn \V f\V_{L^2}. 
	\end{split}
	\end{equation} Similarly, for any integer $m \ge 1$, we have \begin{equation}\label{eq:Hm_Hilbert_alpha}
	\begin{split}
	\V \tilde{H}^{(n)}(f)\V_{H^m} \le C_m n \V f\V_{H^m}. 
	\end{split}
	\end{equation}
\end{lemma}

\begin{proof}
	Without loss of generality, we assume that $f$ is real and consider the real and the imaginary parts of $\tilde{H}^{(n)}_j$ separately. 
	
	Regarding the real part, we need to bound \begin{equation*}
	\begin{split}
	\V \frac{1}{\pi} \int_{\mathbb{R}^+} \frac{-\Re(\zeta^j_{2n}) w + t }{w^2 - 2\Re(\zeta^j_{2n})wt + t^2} f(t) dt \V_{L^2_w},
	\end{split}
	\end{equation*} and note that $\Re(\zeta^j_{2n}) = \cos\left(\frac{j\pi}{n}\right)$. First making the change of variable $t = ws$ and then using the Minkowski inequality, the above quantity is bounded by \begin{equation*}
	\begin{split}
	\frac{1}{\pi} \int_{\mathbb{R}^+} \frac{\left| s - \cos\left(\frac{j\pi}{n}\right) \right|}{s^2 - 2\cos\left(\frac{j\pi}{n}\right)s + 1} \V f(ws) \V_{L^2_w} ds \le C\V f\V_{L^2} \cdot \int_{\mathbb{R}^+} \frac{\left| s - \cos\left(\frac{j\pi}{n}\right) \right|}{s^2 - 2\cos\left(\frac{j\pi}{n}\right)s + 1} \cdot \frac{1}{s^{1/2}} ds.
	\end{split}
	\end{equation*} Then it suffices to show that \begin{equation*}
	\begin{split}
	\int_{\mathbb{R}^+} \frac{\left| s - \cos\left(\frac{j\pi}{n}\right) \right|}{s^2 - 2\cos\left(\frac{j\pi}{n}\right)s + 1} \cdot \frac{1}{s^{1/2}} ds \le  C \ln\left( 1 + \frac{1}{\sin^2\left(\frac{j\pi}{n}\right)} \right)
	\end{split}
	\end{equation*} for some universal $C > 0$. To see this, we consider the case when $\cos(j\pi/n) \ge 0$ (the other case can be treated similarly and is only simpler). In the limit $\cos(j \pi/n) \rightarrow 0^+$, it is easy to see that the integrand is uniformly bounded in $j, n$ by an integrable function and therefore the integrals are uniformly bounded as well. Hence, we may assume that $\cos(j\pi/n) \ge 3/4$, and in this case, we split the integral as \begin{equation*}
	\begin{split}
	\left[ \int_{0\le s < \frac{1}{2}\cos\left(\frac{j\pi}{n}\right) } + \int_{\frac{1}{2} \cos\left(\frac{j\pi}{n}\right) \le s < \frac{3}{2}\cos\left(\frac{j\pi}{n}\right) } + \int_{\frac{3}{2}\cos\left(\frac{j\pi}{n}\right) \le s} \right]\frac{\left| s - \cos\left(\frac{j\pi}{n}\right) \right|}{\left(s - \cos\left(\frac{j\pi}{n}\right)\right)^2 + \sin^2\left(\frac{j\pi}{n}\right)} \cdot \frac{1}{s^{1/2}} ds,
	\end{split}
	\end{equation*} and in the first and third regions, the integrals are uniformly bounded whenever $\cos(j\pi/n) \ge 3/4$. In the second region, after a change of variable, we have a bound \begin{equation*}
	\begin{split}
	C\int_{0 \le s \le\frac{1}{2} \cos\left(\frac{j\pi}{n}\right) } \frac{s}{s^2+\sin^2\left(\frac{j\pi}{n}\right)} \cdot \frac{1}{\left(s + \cos\left(\frac{j\pi}{n}\right)\right)^{1/2}} ds &\le C'\int_{0 \le s \le\frac{1}{2} } \frac{s}{s^2+\sin^2\left(\frac{j\pi}{n}\right)} ds \\
	&\le C \ln\left( 1 + \frac{1}{\sin^2\left(\frac{j\pi}{n}\right)} \right),
	\end{split}
	\end{equation*} which is the bound we wanted. The imaginary part can be treated in a similar way and results in a better estimate, without the logarithmic factor. 
	
	Given this $L^2$-bound for $\tilde{H}^{(n)}_j$, the bound \eqref{eq:L2_Hilbert_alpha} follows simply from summing the estimates over $1 \le |j| \le n$: it only suffices to observe that \begin{equation*}
	\begin{split}
	\sum_{1 \le |j| \le n}\ln\left( 1 + \frac{1}{\sin^2\left(\frac{j\pi}{n}\right)} \right) &= \ln\left( \prod_{1 \le |j| \le n} \left( 1 + \frac{1}{\sin^2\left(\frac{j\pi}{n}\right)} \right)  \right) \\ 
	&\le C \ln\left( \prod_{j=1}^n \left( 1 + \frac{n}{j}\right) \right)  \le C'n,
	\end{split}
	\end{equation*}
	where the last inequality is a consequence of the fact that $$\prod_{j=1}^n\left(1+\frac{n}{j}\right) \le \frac{(2n)!}{(n!)^2} \le C \exp(Cn)$$ for some constant $C>0$ using Sterling's approximation.  
	
	Finally, the $H^m$-bound follows similarly, using the identity \eqref{eq:diff}.
\end{proof}

\subsection{Functional inequalities}

We state and prove the Hardy inequalities. 

\begin{lemma}[Hardy inequalities]
	For any $f \in H^s(\mathbb{R})$ with $s \ge 0$, we have \begin{equation*}
	\begin{split}
	\V \frac{1}{z} \int_0^z f(w) dw \V_{H^s(\mathbb{R})} \le C_s \V f \V_{H^s(\mathbb{R})}.
	\end{split}
	\end{equation*} More precisely, for each $0 \le \sigma \le s$, we have \begin{equation}\label{eq:Hardy}
	\begin{split}
	\V \partial_z^\sigma \left( \frac{1}{z} \int_0^z f(w) dw   \right) \V_{L^2(\mathbb{R})} < \frac{2}{2\sigma + 1} \V \partial_z^\sigma f \V_{L^2}. 
	\end{split}
	\end{equation}
\end{lemma}

\begin{proof}
	In the proof, without loss of generality we shall assume that $f \in C^\infty_c([0,\infty))$.
	
	We first consider the case $s = 0$: we have \begin{equation*}
	\begin{split}
	\V \frac{1}{z} \int_0^z f(w)dw\V_{L^2(0,\infty)}^2 &= \int_0^\infty -\partial_z z^{-1} \left( \int_0^z f(w)dw \right)^2 dz \\
	&= 2 \int_0^\infty f(z) \cdot \left(\frac{1}{z} \int_0^z f(w)dw\right) dz \\
	&<  2 \V \frac{1}{z} \int_0^z f(w)dw\V_{L^2(0,\infty)} \cdot \V f \V_{L^2(0,\infty)},
	\end{split}
	\end{equation*} which establishes the statement in this case, noting that the strict equality forces $f(z) = c_1 z^{c_2}$ for some constants $c_1, c_2$, which never belongs to $L^2(\mathbb{R})$. Next, take $s = \sigma = 1$, and from \begin{equation*}
	\begin{split}
	\partial_z \left(z^{-1}\int_0^z f(w)dw \right)  = z^{-2} \left( zf(z) - \int_0^z f(w)dw \right),
	\end{split}
	\end{equation*} we proceed similarly as before: \begin{equation*}
	\begin{split}
	\V \partial_z \left( \frac{1}{z} \int_0^z f(w) dw   \right) \V_{L^2(0,\infty)}^2 &= 
	-\frac{1}{3}\int_0^\infty \partial_z z^{-3} \left( zf(z) - \int_0^z f(w)dw \right)^2 dz \\
	&< \frac{2}{3} \V \partial_z \left( \frac{1}{z} \int_0^z f(w) dw   \right) \V_{L^2(0,\infty)} \cdot \V \partial_z f\V_{L^2(0,\infty)}. 
	\end{split}
	\end{equation*} The argument for the case $s > 1$ is strictly analogous. We omit the proof. 
\end{proof}

\begin{lemma}[Hardy-type inequalities]
	For any $f \in H^s(\mathbb{R}^+)$ with $s \ge 0$ and $0 < \alpha \le 1$, we have \begin{equation*}
	\begin{split}
	\V z^{-\frac{1}{\alpha}} \int_0^z f(w) \cdot \frac{1}{\alpha}  w^{\frac{1-\alpha}{\alpha}}   dw \V_{H^s(\mathbb{R}^+)} \le C_s \V f \V_{H^s(\mathbb{R}^+)}
	\end{split}
	\end{equation*} with a constant uniform for $0 < \alpha \le 1$.  More precisely, for each $0 \le \sigma \le s$, we have \begin{equation}\label{eq:Hardy_alpha}
	\begin{split}
	\V \partial_z^\sigma \left( z^{-\frac{1}{\alpha}}  \int_0^z f(w) \cdot \frac{1}{\alpha}  w^{\frac{1-\alpha}{\alpha}}  dw   \right) \V_{L^2(\mathbb{R}^+)} < \frac{2}{2\sigma + 1} \V \partial_z^\sigma f \V_{L^2(\mathbb{R}^+)}. 
	\end{split}
	\end{equation}
\end{lemma}
\begin{proof}
	The proof is strictly analogous to that of the usual Hardy inequalities. Let us restrict ourselves to the case $\sigma = 0$. Then, we simply write out \begin{equation*}
	\begin{split}
	\V   z^{-\frac{1}{\alpha}}  \int_0^z f(w) \cdot \frac{1}{\alpha}  w^{\frac{1-\alpha}{\alpha}}  dw  \V_{L^2(\mathbb{R}^+)}^2 &= \frac{1}{1-2/\alpha} \int_0^\infty  
	\partial_z \left( z^{1 - \frac{2}{\alpha}} \right) \cdot \left( z^{-\frac{1}{\alpha}} \int_0^z f(w) \cdot \frac{1}{\alpha}  w^{\frac{1-\alpha}{\alpha}}   dw  \right)^2 dz \\
	&= - \frac{2}{1-2/\alpha} \int_0^\infty \left( z^{-\frac{1}{\alpha}} \int_0^z f(w) \cdot \frac{1}{\alpha}  w^{\frac{1-\alpha}{\alpha}}   dw  \right) \cdot \frac{1}{\alpha} f(z)dz \\
	&= \frac{2}{2-\alpha} \int_0^\infty \left( z^{-\frac{1}{\alpha}} \int_0^z f(w) \cdot \frac{1}{\alpha}  w^{\frac{1-\alpha}{\alpha}}   dw  \right) \cdot f(z) dz
	\end{split}
	\end{equation*} and then applying the Cauchy-Schwartz inequality finishes the proof. 
\end{proof}

\end{document}